\documentclass{amsart}
\usepackage{amssymb, latexsym, euscript}
\usepackage[usenames,dvipsnames,svgnames,table]{xcolor}
\usepackage{lineno}
\def\sA{{\mathfrak A}}      
\def\sD{{\mathfrak D}}   \def\sE{{\mathfrak E}}   
   \def\sH{{\mathfrak H}}   
   \def\sK{{\mathfrak K}}   \def\sL{{\mathfrak L}}
\def\sM{{\mathfrak M}}      
      \def\sR{{\mathfrak R}}
      
      \def\sX{{\mathfrak X}}
\def\sY{{\mathfrak Y}}

\def\st{{\mathfrak t}}

   \def\dN{{\mathbb N}}

\def\cD{{\mathcal D}}

\def\bB{{\mathbf{B}}}

\def\bLR{{\mathbf{L}}}

\def\bm\chi{\mbox{\boldmath$\chi$}}
\def\half{{\frac{1}{2}}}

\def\col{{\rm col\,}}

\def\ker{{\rm ker\,}}
\def\ran{{\rm ran\,}}
\def\cran{{\rm \overline{ran}\,}}
\def\dom{{\rm dom\,}}
\def\mul{{\rm mul\,}}

\def\cdom{{\rm \overline{dom}\,}}
\def\clos{{\rm clos\,}}
\def\dim{{\rm dim\,}}

\let\xker=\ker \def\ker{{\xker\,}}

\DeclareMathOperator{\hplus}{\, \widehat + \,}

\newtheorem{theorem}{Theorem}[section]
\newtheorem{proposition}[theorem]{Proposition}
\newtheorem{corollary}[theorem]{Corollary}
\newtheorem{lemma}[theorem]{Lemma}
\newtheorem{definition}[theorem]{Definition}

\theoremstyle{definition}

\newtheorem{remark}[theorem]{Remark}

\numberwithin{equation}{section}

\begin{document}
\title[Lebesgue type decompositions]
{Complementation and Lebesgue type \\  decompositions
of linear operators \\ and relations}

\author[S.~Hassi]{S.~Hassi}
\author[H.S.V.~de~Snoo]{H.S.V.~de~Snoo}

\address{Department of Mathematics and Statistics \\
University of Vaasa \\
P.O. Box 700, 65101 Vaasa \\
Finland} \email{sha@uwasa.fi}

\address{Bernoulli Institute for Mathematics, Computer Science and Artificial Intelligence \\
University of Groningen \\
P.O. Box 407, 9700 AK Groningen \\
Nederland}
\email{h.s.v.de.snoo@rug.nl}

\date{}
\thanks{Part of this paper was completed during the Workshop ``Spectral theory of differential operators in quantum theory'' at the Erwin Schr\"odinger International Institute for Mathematics and Physics (ESI) (Wien, November  7--11, 2022). The ESI support is gratefully acknowledged.}
%\translator{}

%\date{ \today; Filename: \jobname.}


\begin{abstract}
In this paper a new general approach is developed to construct and study
Lebesgue type decompositions of linear operators $T$ in the Hilbert space setting.
The new approach allows to introduce an essentially wider class
of Lebesgue type decompositions
than what has been studied in the literature so far.
The key point is that it allows a nontrivial interaction
between the closable and the singular components of $T$.
The motivation to study such decompositions comes
from the fact that they naturally occur in the
corresponding Lebesgue type decomposition for pairs of quadratic forms.
The approach built in this paper uses so-called complementation in Hilbert spaces,
a notion going back to de Branges and Rovnyak.
\end{abstract}

\maketitle

\section{Introduction}

The usual Lebesgue decomposition of measures has inspired the study of similar
decompositions of, for instance, pairs of positive operators and semibounded forms; see \cite{An,AnZ1986,HSeSnLeb2009,S3}.
In the context of linear operators or linear relations such decompositions
can be seen as the source for all the other Lebesgue type decompositions.
It should be noted that the standard Lebesgue decomposition of a pair
of positive measures can be obtained
as a special case of the Lebesgue decomposition of a pair of nonnegative forms;
for details, see \cite{HSeSnLeb2009}.

\medskip

In this paper a new general type of decomposition of linear operators
and, more generally, of linear relations is introduced and explained which
allows a nontrivial interaction between the closable component and the
singular component.
This work is inspired by the Lebesgue type decompositions for a pair of forms
which have been studied in \cite{HSeSnLeb2009}.
It turned out that in the Lebesgue type decompositions of
nonnegative forms $\st$ written as the additive sum $\st=\st_1+\st_2$,
where $\st_1$ is a closable form and $\st_2$ is a singular form,
the components $\st_1$ and $\st_2$
need not in general be singular with respect to each other.
In the setting of measures this corresponds to the situation,
where the absolute continuous and the singular component
are not mutually singular with respect to each other.
On the other hand, all Lebesgue type decompositions
of a nonnegative quadratic form can be derived by
introducing a so-called representing map $Q:\sH\to\sK$
(here $\sH$ and $\sK$ are Hilbert spaces) for the form $\st$:
$\st[h,k]=(Qh,Qk)$, $h,k\in\dom \st=\dom Q$;
a detailed study of representing maps will appear in \cite{HS2022c}.
A key fact in the connection of quadratic forms is
that the components $\st_1$ and $\st_2$ generate
a nonnegative contraction $K$ acting in the range space $\sK$ of $Q$,
such that $\st_1[h,k]=((I-K)^{\half}h,(I-K)^{\half}k)$,
$\st_2[h,k]=((K^{\half}h,K^{\half}k)$,
and then one can prove the following general formula
\[
 (\st_1:\st_2)[h,k]=(((I-K):K)Qh,Qk), \quad h,k\in\dom \st=\dom Q,
\]
where ``$:$'' stands for the parallel sum of the involved components; see \cite{HS2022c}.
Recall from \cite[Proposition~2.10]{HSeSnLeb2009}
that the forms $\st_1$ and $\st_2$ are mutually singular  precisely when $\st_1:\st_2=0$,
while $(I-K):K=0$ if and only if $K$ is an orthogonal projection,
which is equivalent to the intersection $\ran (I-K)\cap \ran K\neq \{0\}$ being nontrivial;
further details can be found in \cite{HS2022c}.

\medskip

The new approach developed here to analyse this phenomenon
on the side of linear operators and their Lebesgue type decompositions
allowing such an interaction between the closable and singular parts of $T$
is built in this paper by using the notion of
complementation going back to de Branges and Rovnyak.
This leads to several new results on Lebesgue type decompositions
of unbounded operators and linear relations.
In particular, the results generalize recent results obtained
in the case of orthogonal operator range decompositions in \cite{HSS2018,HS2022a}.
For instance, among the set of all such Lebesgue type decompositions
there is still a unique decomposition,
whose regular part in this new setting continues to be maximal;
it is called the Lebesgue decomposition of linear operators, cf. \cite{HS2022a}.

\medskip

Let $T \in \bLR(\sH, \sK)$, that is, $T$  is a linear operator or relation
from a Hilbert space $\sH$ to a Hilbert space $\sK$.
Denote by $T^{**}$ the closure of $T$ (as a graph in the Cartesian product $\sH \times \sK$);
moreover,
$\mul T^{**}$ stands for the linear space of all $g \in \sK$ for which $\{0,g\} \in T^{**}$.
For $T^{**}$ there are two extreme cases:  the \textit{closable} case is defined by the equality
$\mul T^{**}=\{0\}$, that is,
the closure $T^{**}$ is an operator, and the \textit{singular} case is defined by the equality
$T^{**}=\dom T^{**} \times \mul T^{**}$,
that is,  $T^{**}$ is the Cartesian product of closed linear subspaces of $\sH$ and $\sK$.
In general, a linear relation is neither closable nor singular.
However, every  $T \in \bLR(\sH, \sK)$ has a sum decomposition
 $T=T_1+T_2$ of the form
\begin{equation*}%\label{pu1}
 T=\big\{ \{f,g\} \in \sH \times \sK:\, g=g_1+g_2,  \,\{f, g_1\} \in T_1, \, \{f,g_2\} \in T_2 \big\},
\end{equation*}
where $T_1, T_2 \in \bLR(\sH, \sK)$
with $\dom T=\dom T_1=\dom T_2$,
while $T_1$ is closable and $T_2$ is singular.   Such a sum decomposition  $T=T_1+T_2$
is called an \textit{orthogonal Lebesgue type decomposition} of $T$
if  the Hilbert space $\sK$ is the
orthogonal sum of the closed linear subspaces $\sX$ and $\sY$ of $\sK$,
such that $\ran T_1 \subset \sX$ and  $\ran T_2 \subset \sY$.
The usual Lebesgue decomposition of $T$ is an example
of an orthogonal Lebesgue type decomposition.
In the case of an orthogonal Lebesgue type decomposition
there is only trivial interaction between the summands $T_1$ and $T_2$
as their ranges are orthogonal. Orthogonal Lebesgue type decompositions
have been studied in \cite{HSS2018, HS2022a}, extending
earlier work of Izumino \cite{I89a, I89b, I93}.

\medskip

In this paper more general, pseudo-orthogonal,
decompositions  $T=T_1+T_2$ will be introduced.
The decomposition of the space $\sK$ will be based on
a pair of complemented operator range spaces  $\sX$ and $\sY$
with inner products $(\cdot, \cdot)_\sX$ and $(\cdot, \cdot)_\sY$
which are contained contractively in $\sK$; such spaces were
introduced by de Branges and Rovnyak,
see \cite{dB, DR90, DR91, FW}.
It is assumed that $\sX$ and $\sY$ are generated by nonnegative
contractions $X,Y \in \bB(\sK)$ for which
\[
 \|h\|_\sK^2=\|Xh\|_\sX^2+\|Yh\|_\sY^2, \quad h \in \sK,
\]
and this is equivalent to the condition $X+Y=I$;
moreover, this condition automatically leads to $\sK=\sX+\sY$.
For the sum decomposition $T=T_1+T_2$ which satisfies $\dom T=\dom T_1=\dom T_2$,
while $T_1$ closable and $T_2$ singular,
one now requires that
$\ran T_1 \subset \sX$ and  $\ran T_2 \subset \sY$,
in which case the decomposition of any $g \in \ran T$ as $g=g_1+g_2$
with $g_1 \in \ran T_1$ and $g_2 \in \ran T_2$
leads to the inequality
\[
\| g\|_\sH^2 \leq \|g_1\|_\sX^2 + \|g_1\|_\sY^2.
\]
The further condition
is that instead of this inequality there is the  Pythagorean equality
\[
\| g\|_\sH^2 = \|g_1\|_\sX^2 + \|g_1\|_\sY^2.
\]
In this case one speaks of a
\textit{pseudo-orthogonal Lebesgue type decomposition} of $T$.
In the orthogonal case  the closed linear subspaces $\sX$ and $\sY$
are isometrically contained in $\sK$ and $\ran T_1 \perp \ran T_2$,
so that the Pythagorean equality
is automatically satisfied.
The new feature with complemented operator range spaces $\sX$ and $\sY$
which are contractively contained
in the original Hilbert space $\sK$ is that there is an overlapping space $\sX \cap \sY$;
moreover $\sX \cap \sY$ contains the intersection $\ran X \cap \ran Y$.
This overlapping space has consequences for the interaction of the components
$T_1$ and $T_2$: it may now happen that  $\cran T_1 \cap \cran T_2$  is nontrivial,
or even that   $\ran T_1 \cap \ran T_2$ is nontrivial.

\medskip

It will be shown that the pseudo-orthogonal Lebesgue type decompositions
of a linear relation $T \in \bLR(\sH, \sK)$ are all of the form
\begin{equation*}%\label{pu2}
 T=T_1+T_2 \quad \mbox{with} \quad T_1=(I-K)T, \quad T_2=KT,
\end{equation*}
where $K \in \bB(\sK)$ is a nonnegative contraction for which $(I-K)T$
is closable and $KT$ is singular.
The pseudo-orthogonal Lebesgue type decomposition
is orthogonal precisely if $K$ is an orthogonal projection.
The usual Lebesgue decomposition
$T=T_{\rm reg} +T_{\rm sing}$ is orthogonal and it is uniquely defined by the property
that $T_{\rm reg}$ is the largest closable part of $T$ among
all pseudo-orthogonal decompositions of $T$.
Furthermore, there is a characterization of the situation
where the Lebesgue decomposition
is the only pseudo-orthogonal Lebesgue type decomposition of $T$.
In the special case that $T$ is an operator range relation, i.e.,
\[
 T=\big\{ \{\Phi \eta, \Psi \eta \} :\, \eta \in \sE \big\},
\]
where $\Phi \in \bB(\sE, \sH)$, $\Psi \in \bB(\sE, \sK)$,
and $\sE$, $\sH$, and $\sK$ are Hilbert spaces,
the pseudo-orthogonal Lebesgue type decompositions
of $T$ translate into the so-called
pseudo-orthogonal Lebesgue type decompositions
of the operator $\Psi$ in terms of $\Phi$; for the orthogonal case
and the notion of Radon-Nikodym derivative,
see \cite{HS2022a}.

\medskip

The contents of the paper are now described.
A short introduction to pairs of complemented operator range spaces
can be found in Section \ref{compl}.
This section is modeled on the relevant appendix in \cite{DR90}.
General pseudo-orthogonal decompositions
of a linear relation are introduced in Section \ref{overlap},
where also some material on the occurence of overlapping in a
pseudo-orthogonal decomposition  can be found.
For a relation $T \in \bLR(\sH, \sK)$
and a nonnegative contraction $R \in \bB(\sK)$
there are a number of criteria in Section \ref{Regsing}
under which the product relation $RT$ is regular or singular.
In Section \ref{lebestype}
the previous characterizations are used to study
the pseudo-orthogonal Lebesgue type decompositions of a linear
relation $T$.
The particular case where $T$ is an operator range relation
is briefly reviewed in Section \ref{OpRan}; see \cite{HS2022a}
for the orthogonal case and the corresponding Radon-Nikodym derivatices.

\medskip

The Lebesgue decomposition for measures and the associated
Radon-Nikodym derivatives for their absolutely continuous parts  have seen many
generalizations to more abstract settings.  At this stage it suffices
to mention the work of Dye \cite{Dy1952} and Henle \cite{H1972}.
The second half of the seventies saw the work of Ando \cite{An}
for pairs of nonnegative operators and the work of Simon \cite{S3} for
nonnegative forms.
This lead to many papers devoted to related contexts,
such as $C^*$-algebras and the theory of positive maps; see, for instance,
the references in \cite{AnZ1986, GK2009, HSeSnLeb2009, Kosh},
and note also, more recently,  \cite{CH2022, V2017},
and e.g. the construction of Lebesgue decomposition of
non-commutative measures in multi-variable setting
into absolutely continuous and singular parts via
Lebesgue decompositions for quadratic forms and
via reproducing kernel space techniques; see \cite{JM2019, JM2022}.
Shortly after the papers of Ando and Simon appeared
the work of Jorgensen \cite{J80} and \^Ota \cite{KO, Ota84, Ota87},
which was devoted to the decompositions of  linear operators.
This context (linear operators and also linear relations)
was taken up in  \cite{HSSS2007} and later in \cite{HSS2018, HS2022a}.
The Lebesgue type decompositions in those papers were orthogonal,
whereas in the present paper the pseudo-orthogonal case is dealt with.
Decomposition results in the context of forms,
based on a Hilbert space decomposition similar to the de Branges-Rovnyak decomposition
(as worked out at the end of Section \ref{compl}),
will appear in \cite{HS2022c} and include the results in Simon \cite{S3}.
In the pseudo-orthogonal  decompositions the notion of overlapping spaces appears
in a natural way.
Furthermore, the pseudo-orthogonal situation for a pair of nonnegative bounded operators
(as in \cite{An}) and for a pair of forms on a linear space (as in \cite{HSeSnLeb2009})
can also be treated in the context of the associated linear relations; this is connected
to the recent work by Z. Sebesty\'en, Z. Sz\"ucs, Z. Tarcsay, and T. Titkos
\cite{STT, ST2013, S2013, ST, T2013, T2016, T2020, Tit2012, Tit2013}.

\section{Pseudo-orthogonal decompositions}
\label{compl}

This section provides a short review of pseudo-orthogonal
decompositions of a Hilbert space.
 These decompositions involve linear subspaces of such a Hilbert space,
which are generated by a pair of nonnegative contractions. In such
decompositions there is in general an overlapping of the summands;
 \cite{dB, DR90, DR91, FW}. At the end of this section there is a brief discussion
of an analogous overlapping decomposition; see \cite{HS2022c}.

\medskip

This review begins with the notion of an operator range space.
Let $\sK$ be a Hilbert space and let $A \in \bB(\sK)$ be a
nonnegative contraction.
Provide the range $\sA=\ran A^\half$, a subspace of $\sK$,
 with the inner product
\begin{equation}\label{ipe}
 (A^\half \varphi, A^\half \psi)_\sA=(\pi \varphi, \pi \psi)_\sK, \quad \varphi, \psi \in \sK,
\end{equation}
where $\pi$ is the orthogonal projection in $\sK$ onto
$\cran A^\half=(\ker A^\half)^\perp$. Note that it follows from \eqref{ipe} that the mapping
\begin{equation}\label{ipee}
\varphi \mapsto A^\half \varphi, \quad \varphi \in \cran A^\half,
\end{equation}
is unitary from $\cran A^\half$ onto $\sA$.
Clearly $\sA$ with this inner product is a Hilbert space.
Since $A$ is a contraction, one has
\[
\| A^\half \varphi\|_\sK=\| A^\half \pi \varphi\|_\sK\leq \|\pi \varphi \|_\sK, \quad \varphi \in \sK,
\]
and, hence, the identity \eqref{ipe} shows that
\begin{equation}\label{ipc}
 \|A^\half \varphi \|_\sA \geq \| A^\half \varphi \|_\sK, \quad \varphi \in \sH.
\end{equation}
It is a  consequence of \eqref{ipe} that
\begin{equation}\label{ipcc}
 (A^\half \varphi, A \psi)_\sA=(A^\half \varphi, \psi)_\sK, \quad \varphi, \psi \in \sK,
\end{equation}
which shows that the linear space $\ran A$ is dense in the Hilbert space $\sA$.
Moreover, \eqref{ipcc} leads to the useful identities
 \begin{equation}\label{ip}
 (A \varphi, A \psi)_\sA=(A \varphi, \psi)_\sK  \quad \mbox{and}  \quad
  (A^\half \varphi, A A^\half \psi)_\sA=(A^\half \varphi, A^\half \psi)_\sK, \quad \varphi, \psi \in \sK.
\end{equation}
It is clear that $A$ maps $\sA=\ran A^\half$ into itself;
in fact, it can be seen from \eqref{ip} that $A$ is nonnegative in $\sA$
and that $A$ maps $\sA$  contractively into itself.
Note that if $A$ is an orthogonal projection in $\sK$, then $\pi=A$ and
\begin{equation}\label{iiippp}
(A^\half \varphi, A^\half \psi)_\sA=(A^\half \varphi, A^\half \psi)_\sK, \quad \varphi, \psi \in \sK,
\end{equation}
so that the inner product $(\cdot, \cdot)_\sA$ on $\ran A^{\half}=\ran A$
coincides with the inner product of $\sK$.

\medskip

The equality \eqref{ipe} and the inequality \eqref{ipc} can be formalized.
Recall that a  linear subspace $\sM$ of a Hilbert space $\sK$
is called a \textit{contractive operator range space}, when
$\sM$ has an inner product $(\cdot, \cdot)_\sM$,
such that
\begin{enumerate}\def\labelenumi {\rm (\alph{enumi})}
\item $\|\varphi \|_\sK \leq \| \varphi \|_\sM$, $\varphi \in \sM$;
\item $\sM$ with the inner product $(\cdot, \cdot)_\sM$ is a Hilbert space.
\end{enumerate}
It is clear that the space $\sA$ above
is an example of a contractive operator range space.
In fact, it is the only example; see \cite{HS2022a}.

\medskip

The interest in this section is in pairs of nonnegative contractions $X, Y \in \bB(\sK)$
with the connecting property $X+Y=I$.
 For the convenience of the reader some simple, but useful facts are presented.

\begin{lemma}\label{Klemma}
Let $\sK$ be a Hilbert space and let $X, Y \in \bB(\sK)$
be nonnegative contractions with $X+Y=I$.
Then $XY=YX \in \bB(\sK)$ is a nonnegative contraction with
\begin{equation}\label{KK}
 \ker XY  = \ker X \oplus \ker Y.
\end{equation}
Moreover, the following identities hold:
\begin{equation}\label{K0}
  \left\{ \begin{array}{l}
  \ran X \cap\ran Y=\ran XY, \\
  \ran X^{\half}\cap\ran Y^{\half}=\ran X^{\half}Y^{\half}, \\
  \cran X \cap\cran Y=\cran XY.
\end{array}
\right.
\end{equation}
  Consequently, each of the following statements
\[
 \ran X \cap \ran Y=\{0\}, \quad  \ran X^\half \cap \ran Y^\half=\{0\},
 \]
{\rm (}or, similarly, with the closures of the ranges{\rm)}
 and, in particular $\ran X \perp \ran Y$ or $\ran X^\half \perp \ran Y^\half$,
 is equivalent to the nonnegative contractions $X$ and $Y$
being orthogonal projections.
\end{lemma}

\begin{proof}
The commutativity of $X$ and $Y$ and of their square roots is clear.
 Hence the nonnegativity of the product $XY$ follows from $XY=X^\half Y X^\half$.
 Note that $\ker X$ and $\ker Y$ are perpendicular in $\sK$.
It is clear that the right-hand side of \eqref{KK}
is contained in the left-hand side.
To show the remaining inclusion let $h \in \ker XY$. Then
$h=Yh+Xh$ with $Yh \in \ker X$ and $Xh \in \ker Y$.

To see the first identity in \eqref{K0}, let $g\in\ran X \cap \ran Y$.
Then clearly one has $g=X h=Y k$ for some $h,k\in\sK$.
Hence $h=Y(h+k) \in \ran Y$ and $g \in \ran XY$. The reverse inclusion is clear.
For the second identity in \eqref{K0}, let $g\in\ran X^\half\cap \ran Y^\half$.
Then one has, similarly,
$g= X^\half h=Y^\half k$ for some $h,k\in\sK$.
Hence
\[
h=Yh+X^\half  Y^\half k \in \ran Y^\half \quad \mbox{and} \quad g \in \ran X^\half Y^\half.
\]
The reverse inclusion is clear.

By taking orthogonal complements in \eqref{KK} one obtains
the third identity in \eqref{K0} for the closures of the ranges.
\end{proof}

Let $\sK$ be a Hilbert space and let $X, Y \in \bB(\sK)$
be nonnegative contractions with $X+Y=I$.
Let  $\sX=\ran X^\half$ and $\sY=\ran Y^\half$ be the corresponding
operator range spaces; cf. \eqref{ipe}.
Then the Hilbert space has a decomposition of the form
\begin{equation}\label{fw}
 \sK=\sX + \sY.
\end{equation}
This can be seen as follows.
By definition one has $\sX \subset \sK$ and $\sY \subset \sK$,
so that the right-hand side of \eqref{fw} is contained in the left-hand side.
For the converse, observe that  for all $h \in \sK$ one has
$h=Xh + Yh$ with $Xh \in \sX$ and $Yh \in \sY$,
which gives $\sK \subset \sX+\sY$.
The intersection $\sL=\sX \cap \sY$ is called
the \textit{overlapping space} of the Hilbert spaces $\sX$ and $\sY$
with respect to the decomposition \eqref{fw}.
It is characterized in the following lemma.

\begin{lemma}\label{overlapp}
Let $\sK$ be a Hilbert space and let $X, Y \in \bB(\sK)$
be nonnegative contractions with $X+Y=I$.
 The overlapping space $\sL=\sX \cap \sY$ is an operator range space
associated with $X^\half Y^\half$, whose inner product satisfies
\begin{equation}\label{ipover}
 (\varphi,\psi)_\sL=(\varphi,\psi)_\sX+(\varphi,\psi)_\sY,
 \quad \varphi,\psi \in \sX \cap \sY.
\end{equation}
\end{lemma}

\begin{proof}

The overlapping $\sL=\sX \cap \sY$ in \eqref{fw} is a linear space given  by
$
 \sL =\ran X^\half Y^\half,
$
as follows from Lemma \ref{Klemma}.
To see \eqref{ipover} first observe for $h,k \in \sK$ that the identity $X+Y=I$ gives
\begin{equation}\label{ipover1}
 (h,k)_\sK  =(Yh,k)_\sK+(Xh,k)_\sK  =(Y^\half h, Y^\half k)_\sK+ (X^\half  h, X^\half  k)_\sK.
\end{equation}
Note that $Y^\half \cran X^\half Y^\half \subset \cran X^\half$
and $X^\half \cran X^\half Y^\half \subset \cran Y^\half$.
Hence, if in \eqref{ipover1} one takes
$h,k \in \cran X^\half Y^\half$, then it follows that
\[
 (X^\half Y^\half h, X^\half Y^\half k)_\sL  =(X^\half Y^\half h, X^\half Y^\half k)_\sX
                           +(Y^\half X^\half h, Y^\half X^\half k)_\sY.
\]
Moreover, it is clear that the last identity holds for all $h,k \in \sK$.
Therefore, the inner product on $\sL$ satisfies \eqref{ipover}.
\end{proof}

Let $\sK$ be a Hilbert space and let $X, Y \in \bB(\sK)$
be nonnegative contractions with $X+Y=I$.  Provide the Cartesian product
$\sX \times \sY$ with the inner product generated by $\sX$ and $\sY$, respectively.
 Define the column operator $V$ from $\sK$ to $\sX \times \sY$ by
\begin{equation}\label{vstar00}
V =\col{ (X,Y)}
= \left\{ \left\{ h, \begin{pmatrix} Xh \\Yh\end{pmatrix} \right\} :\,  h \in \sK \right\}.
\end{equation}
 The operator $V$ is clearly isometric, since
\begin{equation}\label{pyth1}
\begin{split}
  \|Xh\|^2_\sX+ \|Yh\|^2_\sY &=(Xh,h)_\sK+(Yh,h)_\sK =((X+Y)h,h),
  \quad h \in \sK,
 \end{split}
\end{equation}
cf. \eqref{ip}.
Hence, $V$ is a closed operator and $\ran V$ is closed.
In general, the isometry $V$ does not map onto $\sX \times \sY$.

\begin{proposition}\label{DrRoN}
Let $\sK$ be a Hilbert space, let $X, Y \in \bB(\sK)$
be nonnegative contractions, and assume that $X+Y=I$.
Let the column operator $V$ be given by \eqref{vstar00}.
Then the adjoint mapping $V^*$ from $\sX \times \sY$ to $\sK$ is
a partial isometry, given by
 \begin{equation}\label{vstar0}
 V^* \begin{pmatrix} f \\ g \end{pmatrix} =f+g, \quad f \in \sX, \,\,g \in \sY.
\end{equation}
 Consequently, for all $f \in \sX$ and $g \in \sY$, there is the inequality
\begin{equation}\label{inek}
 \|f+g\|_\sK^2 \leq \|f\|_\sX^2+\|g\|_\sY^2,
\end{equation}
with equality in \eqref{inek} if and only  if
 $f=Xh$ and $g=Yh$ for some $h \in \sK$.
\end{proposition}

\begin{proof}
 A simple calculation gives for all $f \in \sX$, $g \in \sY$,
and $h \in \sK$ that
\[
\begin{split}
\left( V^* \begin{pmatrix} f \\ g \end{pmatrix} , h \right)_\sK
&=\left( \begin{pmatrix} f \\ g \end{pmatrix} , V h \right)_{\sX \times \sY}
=\left( \begin{pmatrix} f \\ g \end{pmatrix} ,
          \begin{pmatrix} Xh \\Yh\end{pmatrix} \right)_{\sX \times \sY} \\
&=(f, Xh)_\sX+(g, Yh)_\sY=(f, h)_\sK + (g, h)_\sK \\&=(f+g, h)_\sK,
\end{split}
\]
which follows from \eqref{ipe}; the identity shows \eqref{vstar0}.
Since $V$ is an isometry, $V^*$ is partially isometric
and, in particular, $V^*$ is contractive.
Moreover, according to \eqref{fw}, the mapping $V^*$ is onto.
Thus \eqref{vstar0} implies \eqref{inek}.
Finally, there is equality in \eqref{inek} if and only if $\binom{f}{g} \in \ran V$.
\end{proof}

The connection between the overlapping space $\sL=\sX \cap \sY$ and
the range of the isometry $V$ is now clear.

\begin{proposition}\label{ranv}
The isometry $V$ satisfies
\begin{equation}\label{vstar0++}
\begin{split}
   (\ran V)^\perp
  & =\left\{ \begin{pmatrix} -X^\half Y^\half k \\ X^\half Y^\half k \end{pmatrix} :\,
   k \in \cran XY  \right\}.
 \end{split}
 \end{equation}
Moreover, $V$ is surjective if and only if $X$ and $Y$ are orthogonal projections.
\end{proposition}

\begin{proof}
It is clear that $(\ran V)^\perp=\ker V^*$ and \eqref{vstar0} shows that
 \begin{equation}\label{vstar000}
\begin{split}
   \ker V^*
  &=\left\{ \begin{pmatrix} \varphi \\ -\varphi \end{pmatrix} :\, \varphi \in \sL \right\}. \\
  \end{split}
\end{equation}
Now apply Lemma \ref{overlapp} to obtain the assertion \eqref{vstar0++}.
In particular, \eqref{vstar000}
and the isometric property of $V$, cf. \eqref{pyth1}, show that $V$ is surjective if and only if $\sL=\{0\}$.
The conclusion follows from Lemma \ref{overlapp}.
\end{proof}

Recall that $X$ and $Y$ act as nonnegative contractions in $\sX$ and $\sY$, respectively.
 The next corollary presents the orthogonal projection $VV^*$ as a common
dilation in the Hilbert space $\sX \times \sY$ for this pair of nonnegative contractions; cf. \cite{RN}.

\begin{corollary}\label{vpro}
 The orthogonal projection $VV^*$ onto $\ran V$
is given by
\[
 VV^*\begin{pmatrix} f \\ g \end{pmatrix}
  =\begin{pmatrix} X & X  \\Y  &Y \end{pmatrix} \begin{pmatrix} f \\ g \end{pmatrix},
 \quad f \in \sX, \quad g \in \sY.
\]
\end{corollary}

The terminology in the following definition will be used in the rest of this paper.

\begin{definition}\label{ortha}
Let $\sX$ and $\sY$ be linear subspaces of the Hilbert space $\sK$.
Then $\sK$ is said to have
a \textit{pseudo-orthogonal decomposition}
$\sK=\sX+\sY$ if
 \begin{enumerate}\def\labelenumi {\rm (\alph{enumi})}
\item $\sX$ and $\sY$ are contractive operator range spaces
that are contractively contained in $\sK$;
\item the corresponding nonnegative contractions $X$ and $Y$  satisfy $X+Y=I$.
\end{enumerate}
\end{definition}

Recall that the condition $X+Y=I$ is equivalent to the condition
\begin{equation*}%\label{pyth}
 \|h \|_\sK^2=\|Xh\|_\sX^2+\|Yh\|_\sY^2, \quad h \in \sK,
\end{equation*}
cf. \eqref{pyth1} and Theorem \ref{DrRoN}. Moreover,
if $X$ and $Y$ in Definition \ref{ortha}
are orthogonal projections, then the definition reduces
to the usual orthogonal decomposition as the contractive operator range spaces $\sX$
and $\sY$ are closed linear subspaces of $\sK$; cf. Lemma \ref{Klemma} and \eqref{iiippp}.

\medskip

At the end of the section a closely related situation will be reviewed
for nonnegative contractions $X, Y \in \bB(\sK)$ which satisfy $X+Y=I$.
Provide the closed linear subspaces $\sK_1=\cran X$ and $\sK_2=\cran Y$
with the inner product inherited from $\sK$. It is clear that the Hilbert space $\sK$
has the decomposition
\begin{equation}\label{fww}
\sK=\sK_1+\sK_2.
\end{equation}
The intersection $\sK_1 \cap \sK_2$ is called the \textit{overlapping space}
of the Hilbert spaces $\sK_1$ and $\sK_2$ with respect to the decomposition \eqref{fww}.
It is characterized by
\[
\sK_1 \cap \sK_2=\cran X \cap \cran Y=\cran XY;
\]
see Lemma \ref{Klemma}.

\medskip

Note that the column operator  $W=\col (X^\half,Y^\half)$, defined by
\begin{equation}\label{wstar00}
Wh:=\left\{ \left\{ h, \begin{pmatrix} X^\half h \\Y^\half h\end{pmatrix} \right\} :\, h \in \sK \right\},
\end{equation}
is a closed isometric mapping from $\sK$ to $\sK_1 \times \sK_2$.
The mapping $W$ in \eqref{wstar00} is closely related to the mapping $V$ in \eqref{vstar00}.
To see this, first observe that the operator matrix
\begin{equation}\label{uu}
 U=\begin{pmatrix} X^\half & 0 \\ 0 & Y^\half \end{pmatrix}:\,
 \begin{pmatrix} \sK_1 \\ \sK_2 \end{pmatrix} \to  \begin{pmatrix} \sX \\ \sY \end{pmatrix}
\end{equation}
between the indicated Hilbert spaces is a unitary mapping; compare this with the  property
\eqref{ipee} of the operator $A \in \bB(\sK)$ in \eqref{ipe}.
Next observe that $U$ connects the operators $W$ and $V$  via
\begin{equation}\label{uwv}
UW=V.
\end{equation}
Hence the following result is a consequence of Proposition \ref{DrRoN}.

\begin{proposition}\label{DrRo1}
Let $\sK$ be a Hilbert space and let $X, Y \in \bB(\sK)$
be nonnegative contractions with $X+Y=I$.
Let the column operator $W$ be given by \eqref{wstar00}.
 Then the adjoint mapping $W^*$ from $\sK_1 \times \sK_2$ to $\sK$ is
a partial isometry, given by
\begin{equation}\label{wstar0}
 W^* \begin{pmatrix} f \\ g \end{pmatrix} =X^\half f+Y^\half g,
 \quad f \in \sK_1, \, g \in \sK_2.
\end{equation}
Consequently, for all $f \in \sK_1$ and $g \in \sK_2$,  there is the inequality
\[
\|X^\half f+Y^\half g\|_\sK^2 \leq \|f\|^2_\sK+\|g\|_\sK^2,
\]
 with equality in \eqref{wstar0} if and only if $f=X^\half h$ and $g=Y^\half h$
 for some $h \in \sK$.
\end{proposition}

Furthermore, the isometry $W$ is not onto in general and the intersection of
$\cran X$ and $\cran Y$ comes into play.

\begin{proposition} \label{ranw}
The isometry $W$ satisfies
\begin{equation}\label{wstar0+}
(\ran W)^\perp=\ker W^*=
 \left\{\begin{pmatrix}
   -Y^{\half}k \\
  X^{\half}k
 \end{pmatrix}:\,
 k\in\cran XY \right\}.
\end{equation}
Moreover, $W$ is surjetive if and only if $X$ and $Y$ are orthogonal projections.
\end{proposition}

\begin{proof}
The operator $U$ in \eqref{uu} maps $\ker W^*$
in \eqref{wstar0+} onto $\ker V^*$ in \eqref{vstar0++}.
Hence the assertion \eqref{wstar0+} follows from Proposition \ref{ranv}.
The characterization of surjectivity follows from \eqref{uwv}
and Proposition \ref{ranv}.
\end{proof}

Note that $X$ and $Y$ act as nonnegative contractions in $\sK_1$ and $\sK_2$, respectively.
The following corollary presents the orthogonal projection  $WW^*$ as a common
dilation in the Hilbert space $\sK_1 \times \sK_2$ for this pair of nonnegative contractions; cf. \cite{RN}.
It can be seen as a consequence of Corollary \ref{vpro}, since $WW^*=U^*VV^*U$.

\begin{corollary}\label{ww}
The orthogonal projection $WW^*$ onto $\ran W$
is given by
\[
 WW^*\begin{pmatrix} f \\ g \end{pmatrix}
  =\begin{pmatrix} X & X^\half Y^\half  \\Y^\half X^\half  &Y \end{pmatrix}
  \begin{pmatrix} f \\ g \end{pmatrix}, \quad f \in \sK_1, \, g \in \sK_2.
\]
 \end{corollary}

The model involving $\sK_1=\cran X$ and $\sK_2=\cran Y$ is connected
to the de Branges-Rovnyak model involving $\sX=\ran X^\half$ and $\sY=\ran Y^\half$
via the unitary mapping \eqref{uu}.
The present model and the mapping $W$ in \eqref{wstar00} and its properties
will play a role in the Lebesgue type decompositions of a single
semibounded form \cite{HS2022c}.

\section{Pseudo-orthogonal decompositions}\label{overlap}

In this section one can find a brief introduction to sum decompositions of linear
operators or relations from a Hilbert space $\sH$
to a Hilbert space $\sK$ with respect to a so called pseudo-orthogonal
decomposition of $\sK$. First some preliminary properties about
sums of relations are discussed.

\medskip

Let $T_1$ and $T_2$ belong to $\bLR(\sH, \sK)$.
The sum $T_{1}+T_{2} \in \bLR(\sH,\sK)$
is defined by
\begin{equation}\label{aaa0}
 T_1+T_2=\big\{ \{f,f'+f''\} :\, \{f,f'\} \in T_1, \{f,f''\} \in T_2 \big\}.
\end{equation}
With the sum $T=T_{1}+T_{2}$ it is clear that for the domains one has
\[
\dom T=\dom T_{1} \cap \dom T_{2},
\]
while it is straightforward to check for the ranges that there is an inclusion
\[
  \ran T \subset \ran T_{1} + \ran T_{2}.
\]
However, for the multivalued parts there is equality
 \begin{equation}\label{rm1+}
 \mul T=\mul T_{1} + \mul T_{2},
\end{equation}
so that $\mul T_1 \subset \mul T$ and $\mul T_2 \subset \mul T$.

\begin{definition}\label{ssuumm}
The sum in \eqref{aaa0} is said to be \textit{strict} if the sum in \eqref{rm1+} is direct, i.e.,
\[
\mul T_1 \cap \mul T_2=\{0\}.
\]
In other words the sum in \eqref{aaa0} is strict precisely
when the elements $f'$ and $f''$ in \eqref{aaa0} are uniquely determined by the sum $f'+f''$.
\end{definition}

In particular, the sum $T=T_1+T_2$ is strict if either $T_1$ or $T_2$ is an operator.
A variation on the theme of sums is given in the following lemma.

\begin{lemma}\label{opnieuw}
Let $T$, $T_1$, and $T_2$ belong to $\bLR(\sH,\sK)$.
Assume the domain equality
$\dom T=\dom T_1=\dom T_2$ and
the inclusion
\begin{equation}\label{grey1}
 T \subset T_1+T_2.
\end{equation}
Then there is equality $T=T_1+T_2$ in \eqref{grey1} if and only if
\begin{equation}\label{grey2}
\mul T=\mul T_1+\mul T_2.
\end{equation}
Consequently, there is equality in \eqref{grey1} if and only if
\[
\mul T_1 \subset \mul T \quad \mbox{and} \quad \mul T_2 \subset \mul T.
\]
 \end{lemma}

\begin{proof}
By assumption one has $\dom T=\dom (T_1+T_2)$
and it follows from the inclusion \eqref{grey1} that $\mul T \subset \mul (T_1+T_2)$.
Hence, by an observation that goes back to Arens (see \cite[Corollary 1.1.3]{BHS}),
there is equality   $T=T_1+T_2$ if and only if
$\mul (T_1+T_2) \subset \mul T$, i.e., \eqref{grey2} holds.
 \end{proof}

The next corollary illustrates a situation that will be
of interest in the rest of the paper; cf. \cite{HSS2018}.

\begin{corollary}\label{orth+}
Let $T \in \bLR(\sH, \sK)$
and let $X,Y \in \bB(\sK)$ be nonnegative contractions such that $X+Y=I$.
Then
$\dom T=\dom XT=\dom YT$ and, in addition,
 \begin{equation}\label{inclll}
 T \subset XT + YT.
\end{equation}
There is equality $T = XT + YT$ in \eqref{inclll} if and only if
\[
\mul T=X \,\mul T + Y \,\mul T.
\]
Consequently, there is equality  in \eqref{inclll} if and only if
\begin{equation}\label{zwijn0}
X \,\mul T \subset \mul T \quad \mbox{or, equivalently,}
\quad Y \,\mul T \subset \mul T.
\end{equation}
Moreover, in this case
\begin{equation}\label{zwijn}
X \,\mul T \,\cap \,Y\,\mul T=XY\, \mul T;
\end{equation}
thus the sum $T = XT + YT$ is strict in the sense of Definition {\rm \ref{ssuumm}}
if and only if $\mul T \subset \ker XY$.
\end{corollary}

\begin{proof}
These assertions follow from Lemma \ref{opnieuw} except the identity \eqref{zwijn}.
To see \eqref{zwijn} let $h \in X \,\mul T \cap Y\,\mul T$, so that $h=X \varphi=Y \psi$
where $\varphi, \psi \in \mul T$. Now it follows from $(I-Y)\varphi=Y \psi$
that $\varphi =Y(\varphi+\psi)$ with $\varphi+\psi \in \mul T$. Thus $h \in XY \mul T$,
which shows that $X \,\mul T \cap Y\,\mul T \subset XY\, \mul T$.  The reverse inclusion follows
immediately from \eqref{zwijn0}.
\end{proof}

The interest in this paper is in decompositions $T=T_1+T_2$
with linear relations or operators  going from a Hilbert $\sH$ to a Hilbert space $\sK$,
which have a pseudo-orthogonal decomposition $\sK=\sX+\sY$.
Before the formal definition is given, note that
 any element  $\{f,g\} \in T$ can be written as
\[
 \{f,g\}=\{f,g_1+g_2\}, \quad \{f,g_1\} \in T_1, \quad \{f,g_2\} \in T_2, \quad g=g_1+g_2.
\]
If $\ran T_1 \subset \sX$ and $\ran T_2 \subset \sY$, then by  Proposition \ref{DrRoN}
there is the general inequality
\begin{equation}\label{pyth0}
  \|g\|_\sK^2 \leq \|g_1\|_\sX^2 + \|g_2\|_\sY^2.
\end{equation}
In the following definition a special class of such sum decompositions is introduced,
involving a Pythagorean equality in \eqref{pyth0}.

\begin{definition}\label{nonorthsum}
Let $T \in \bLR(\sH, \sK)$ and
assume that  $\sK$ has a \textit{pseudo-orthogonal decomposition}
$\sK=\sX+\sY$ with associated nonnegative contractions $X$ and $Y$
such that $X+Y=I$. Let $T_1$ and $T_2$ belong to $\bLR(\sH,\sK)$,
then the sum
\begin{equation}\label{ltdsum}
  T=T_1+T_2 \quad \mbox{with} \quad \dom T=\dom T_1=\dom T_2,
\end{equation}
is said to be a pseudo-orthogonal decomposition of $T$ connected with the
pseudo-orthogonal decomposition $\sK=\sX+\sY$ {\rm (}or,
equivalently, with the pair of nonnegative contractions
$X$ and $Y$ in $\bB(\sK)$ with $X+Y=I${\rm )} if
\begin{itemize}
\item[(a)] $\ran T_1 \subset \sX$ and $\ran T_2   \subset \sY$;
\item[(b)]  for every $\{f,g\} \in T$ with $\{f,g_1\} \in T_1$, $\{f,g_2\} \in T_2$, $g=g_1+g_2$,
one has
\[
  \|g\|_\sK^2=\|g_1\|_\sX^2 + \|g_2\|_\sY^2.
\]
\end{itemize}
\end{definition}

The definition of pseudo-orthogonal decompositions has an important consequence
for the sum \eqref{ltdsum}; cf. Definition \ref{ssuumm}.

\begin{lemma}\label{nonorthsum1}
Let $T=T_1+T_2$ in \eqref{ltdsum} be a pseudo-orthogonal decomposition.
Then the sum is strict in the sense of Definition {\rm \ref{ssuumm}}.
 \end{lemma}

\begin{proof}
Let $\varphi \in \mul T_1\cap \mul T_2$. Then
$\{0,\varphi\}\in T_1$, $\{0,-\varphi\}\in T_2$, and for the sum
one sees that  $g=\varphi-\varphi=0$.
It follows that $\|\varphi\|_\sX^2 + \|-\varphi\|_\sY^2=0$.
This shows that $\varphi=0$.
Therefore $\mul T_1 \cap \mul T_2=\{0\}$ and the sum $T=T_1+T_2$ is strict.
\end{proof}

The pseudo-orthogonal decompositions in Definition \ref{nonorthsum}
will now be characterized by means of nonnegative contractions in $\bB(\sK)$.

\begin{theorem}\label{prelem0sum}
Let $T \in \bLR(\sH,\sK)$ be a linear relation.
Assume that $K\in\mathbf{B}(\sK)$ is a nonnegative contraction
which satisfies
\begin{equation}\label{sum3}
\mul T= (I-K) \,\mul T + K \, \mul T, \quad \mbox{direct sum},
\end{equation}
and define
\begin{equation}\label{comsum2sum}
 T_1=(I-K)T \quad \mbox{and} \quad T_2=KT.
\end{equation}
Then the sum  $T=T_1+T_2$ in \eqref{ltdsum}
is a pseudo-orthogonal decomposition of $T$,
connected with the pair $I-K$ and $K$
in the sense of Definition {\rm \ref{nonorthsum}}.

Conversely, let the sum $T=T_1+T_2$ in \eqref{ltdsum}
be a pseudo-orthogonal decomposition
of $T \in \bLR(\sH,\sK)$
in the sense of  Definition {\rm \ref{nonorthsum}}.
Then there exists a nonnegative contraction $K \in \bB(\sK)$
for which \eqref{sum3} and  \eqref{comsum2sum} are satisfied.
\end{theorem}

\begin{proof}
Let $T \in \bLR(\sH, \sK)$
and let $K\in\mathbf{B}(\sK)$ be a nonnegative contraction, such that \eqref{sum3} holds.
 By Corollary \ref{orth+} the relations $T_1=(I-K)T$ and $T_2=KT$ in \eqref{comsum2sum}
satisfy $\dom T_1=\dom T_2=\dom T$ and $T \subset T_1+T_2$.
Again by Corollary \ref{orth+} and the identity in \eqref{sum3}
there is the decomposition $T=T_1 + T_2$.
Thus the identities in \eqref{ltdsum} are satisfied.
Since the sum in \eqref{sum3} is direct, the sum $T=T_1+T_2$ is strict.

Now let $\sX$ and $\sY$ be the operator range spaces generated
by the nonnegative contractions $X=I-K$ and $Y=K$, respectively.
Clearly, $\sX$ and $\sY$ form a pair of complemented spaces,
contractively contained in $\sK$, and furthermore
\[
\ran T_1 \subset \ran (I-K) \subset \sX \quad \mbox{and}
\quad \ran T_2 \subset \ran K \subset \sY,
\]
which gives (a) in Definition \ref{nonorthsum}.
In order to check the Pythagorean property (b)
in Definition \ref{nonorthsum},
let $\{f,g\} \in T$. Then $\{f,g\}=\{f, g_1+g_2\}$
with $\{f,g_1\} \in T_1$ and $\{f,g_2\} \in T_2$.
Since the sum $T=T_1+T_2$ is strict, one sees $g_1=(I-K)g$ and $g_2=Kg$.
This implies with \eqref{ip} that
 \[
\begin{split}
 \|g_1\|_\sX^2 + \| g_2\|_\sY^2 &= \|( I-K)g\|_\sX^2+\|Kg\|_\sY^2 \\
 &= ((I-K)g,g)_\sK+(Kg,g)_\sK=\|g\|_\sK^2,
 \end{split}
\]
and the Pythagorean property has been shown.
Hence the conditions in Definition \ref{nonorthsum} are satisfied.

\medskip

Conversely, let $T=T_1+T_2$ be a pseudo-orthogonal decomposition
of $T$ of the form \eqref{ltdsum}.
Let $\{f,g\} \in T$, then by (a) and (b) of Definition \ref{nonorthsum} one has
 for all
 $\{f,g_1\} \in T_1$ and $\{f,g_2\} \in T_2$ with $g=g_1+g_2$, that
  \[
  \|g\|_\sK^2=\|g_1\|_\sX^2 + \|g_2\|_\sY^2.
\]
Thanks to this Pythagorean identity and Proposition \ref{DrRoN}
one obtains $g_1=Xg$ and $g_2=Yg$, which shows $\{f, Xg\}=\{f, g_1\} \in T_1$ and
$\{f, Yg\}=\{f, g_2\} \in T_2$.
 Consequently, one sees the inclusions
\begin{equation}\label{grei3}
XT \subset T_1 \quad \mbox{and} \quad YT \subset T_{2}.
\end{equation}
  By definition,  $\dom T=\dom T_{1}$ and $\dom T=\dom T_2$,
  and  it follows from \eqref{grei3} and \cite[Proposition 1.1.2]{BHS} that
\begin{equation}\label{grei}
 T_{1}=XT \hplus (\{0\} \times \mul T_{1}) \quad \mbox{and}
 \quad   T_{2}=YT \hplus (\{0\} \times \mul T_{2});
\end{equation}
here ``$\hplus$'' stands for the componentwise sum (linear spans) of the graphs.
Observe that it follows from \eqref{grei3} that
$X \mul T \subset \mul T_1$ and $Y \mul T \subset \mul T_2$.
Now let $\{0,h \} \in \mul T_1 \subset \mul T$. Then $X+Y=I$ gives
\[
  h-Xh =Yh \quad \mbox{with} \quad h-Xh \in \mul T_1 \quad \mbox{and} \quad Yh \in \mul T_2.
\]
By Lemma \ref{nonorthsum} it follows that $h=Xh$ and thus $(\{0\} \times \mul T_1) \subset XT$.
Hence, by \eqref{grei} one sees that $T_1=XT$ and, likewise, $T_2=YT$.
Consequently, with $K=Y$ one obtains a nonnegative contraction
$K \in \bB(\sK)$ for which \eqref{sum3} and \eqref{comsum2sum}  hold.
\end{proof}

Note that the nonnegative contraction $K \in \bB(\sK)$ in \eqref{comsum2sum}
is uniquely determined if the relation $T$ is minimal in the sense that $\cran T=\sK$;
 cf. \cite{HSS2018}.
In the case of decompositions of semibounded forms via representing maps
the minimality may be assumed without loss of generality
(see also \cite{HS2022c}).

\medskip

Let $T$, $T_1$, and $T_2$ belong to $\bLR(\sH, \sK)$ and assume that \eqref{ltdsum} holds.
Let the Hilbert space $\sK$ have the
orthogonal decomposition $\sK=\sX \oplus \sY$, where $\sX$ and $\sY$ are
closed subspaces of $\sK$. Then the corresponding nonnegative
contractions $X$ and $Y$, which satisfy $X+Y=I$, are
orthogonal projections onto $\sX$ and $\sY$.
Clearly, the condition {\rm (a)} of Definition {\rm \ref{nonorthsum}}
implies the condition {\rm (b)}.
Therefore the following definition is natural.

\begin{definition}\label{orthsum}
Let $T \in \bLR(\sH,\sK)$ and assume that  $\sK$ has an orthogonal decomposition
$\sK=\sX+\sY$.
Then the sum \eqref{ltdsum} is called an orthogonal sum decomposition of $T$ connected with
the orthogonal decomposition $\sK=\sX+\sY$ {\rm(}or, equivalently,
with the orthogonal projections $I-P$ and $P${\rm )}
if $\ran T \subset \sX$ and $\ran T_2 \subset \sY$.
 \end{definition}

The characterization of orthogonal sum decompositions can be given
as a corollary of Theorem \ref{prelem0sum}; see \cite{HSS2018}, \cite{HS2022a}.

\begin{corollary}
Let  $T \in \bLR(\sH,\sK)$ be a linear relation.
Assume that $P\in\mathbf{B}(\sK)$ is an orthogonal projection which satisfies
\begin{equation}\label{Sum3}
\mul T= (I-P) \,\mul T + P \, \mul T,
\end{equation}
and define
\begin{equation}\label{Comsum2sum}
 T_1=(I-P)T \quad \mbox{and} \quad T_2=PT.
\end{equation}
Then the sum  $T=T_1+T_2$ in \eqref{ltdsum}
is an orthogonal sum decomposition of $T$,
connected with the pair $I-P$ and $P$
in the sense of Definition {\rm \ref{orthsum}}.

Conversely, let the sum $T=T_1+T_2$ in \eqref{ltdsum}
be an orthogonal sum decomposition
of $T \in \bLR(\sH,\sK)$
in the sense of Definition {\rm \ref{orthsum}}.
Then there exists an orthogonal projection $P \in \bB(\sK)$
for which \eqref{Sum3} and \eqref{Comsum2sum} are satisfied.
\end{corollary}

Let $\sH$ be a Hilbert space and let $X,Y \in \bB(\sK)$ be nonnegative contractions
for which $X+Y=I$. The operators $X$ and $Y$ induce Hilbert spaces
$\sX$ and $\sY$ for which $\sK=\sX+\sY$, and Hilbert spaces $\sK_1$ and $\sK_2$
for which $\sK=\sK_1+\sK_2$
each with an overlapping space $\sX \cap \sY$ and $\sK_1 \cap \sK_2$, respectively.
Then one has the inclusions
\[
\left\{\begin{array}{l}
 \sX \cap \sY= \ran X^\half \cap \ran Y^\half
 = \ran X^\half Y^\half \subset  \cran X^\half Y^\half=\cran XY,\\
\ran X \cap \ran Y=\ran XY \subset \cran XY
=\cran X \cap \cran Y =\sK_1 \cap \sK_2,
\end{array}
\right.
\]
see Section \ref{compl}.
For a linear relation $T \in \bLR(\sH,\sK)$
it has been shown in Corollary \ref{orth+}
that with $T_1=XT$ and $T_2=YT$
one has $T=T_1+T_2$ if and only if the linear subspace $\mul T$
is invariant under $X$ or $Y$. Under these circumstances
it is clear that
\[%begin{equation}\label{ranT1ranT2A}
\ran T_1 \cap \ran T_2 \subset \ran X \cap \ran Y=\ran XY.
\]%end{equation}
In order to give a characterization for the intersection $\ran T_1 \cap \,\ran T_2$,
it is convenient to introduce the maximal linear subspace $\sM$ of $\ran T$,
which is mapped back into $\ran T$ by $X$ or by $Y$:
\begin{equation}\label{ell}
 \sM=\big\{\,\eta\in \ran T:\, X\eta \in \ran T\,\big\}=\big\{\,\eta\in \ran T:\, Y\eta \in \ran T\,\big\}.
\end{equation}
Note that $\mul T \subset \sM$ if $T=T_1+T_2$.

\begin{theorem}\label{overl}
Let   $T \in \bLR(\sH,\sK)$ have a decomposition
$T=T_1+T_2$, where $T_1=XT$ and $T_2=YT$
for some nonnegative contractions $X,Y \in \bB(\sK)$ with $X+Y=I$.
Then $\ran T_1 \cap \ran T_2$ is given by
\begin{equation}\label{OL1}
 \ran T_1 \cap \ran T_2 = XY\, \sM,
\end{equation}
where $\sM$ is given in \eqref{ell}.
Consequently, the intersection $\ran T_1 \cap \, \ran T_2$ is nontrivial if and only if
$\sM \cap \ker XY \neq \{0\}$. In particular,
if $X$ or $Y$ is an orthogonal projection,
then $\ran T_1 \cap \, \ran T_2=\{0\}$.
\end{theorem}

\begin{proof}
For the inclusion $(\subset)$ in \eqref{OL1},
assume that $\omega\in \ran T_1 \cap \ran T_2$.
Then for some $\varphi, \psi \in \ran T$ one has
\begin{equation}\label{om01}
 \omega=X\varphi =Y\psi.
\end{equation}
This shows $\psi=X\eta$, where $\eta=\varphi+\psi$;
hence, $\eta \in \ran T$.
Since $\psi=X \eta \in \ran T$,
one sees that $\eta \in \sM$.
Moreover, it follows from \eqref{om01} that
\[
 \omega=Y \psi=YX \eta \in XY \,\sM,
\]
which gives $\ran T_1 \cap \ran T_2\subset XY\, \sM$.

For the inclusion ($\supset$) in \eqref{OL1},
assume that $\eta \in \sM$.
Then, by \eqref{ell},  $\eta\in \ran T$, $X\eta \in \ran T$,
and $Y\eta\in \ran T$.
It follows that $XY\eta \in Y\ran T=\ran T_1$
and that $XY \eta \in X \ran T=\ran T_2$.
Therefore one sees that  $XY\eta \in \ran T_1 \cap \ran T_2$.
Thus, $XY \, \sM \subset  \ran T_1 \cap \ran T_2$.

The final statement follows directly from the identity \eqref{OL1}.
In particular, if $X$ or $Y$ is an orthogonal projection then $XY=0$.
\end{proof}

There is a similar result for the intersection of
$\cran T_1 \cap \cran T_2$ in the presence
of a minimality condition.

\begin{lemma}\label{lap}
Let   $T \in \bLR(\sH,\sK)$ have a decomposition
$T=T_1+T_2$, where $T_1=XT$ and $T_2=YT$
for some nonnegative contractions $X,Y \in \bB(\sK)$ with $X+Y=I$.
Assume in addition that $\cran T=\sK$. Then
\begin{equation}\label{ranT1ranT2}
 \cran T_1 \cap \cran T_2 = \cran XY.
\end{equation}
Consequently, $\cran T_1 \cap \cran T_2=\{0\}$
if and only if $X$ or $Y$ are orthogonal projections.
\end{lemma}

\begin{proof}
Assume $\cran T=\sK$. To see   \eqref{ranT1ranT2} observe the identities
$\cran T_1=\cran X$, $\cran T_2=\cran Y$, and $\cran (XY)T=\cran XY$.
It remains to apply \eqref{K0}, which shows that $\cran X \cap \cran Y=\cran XY$.
For the last statement, see also Lemma \ref{Klemma}.
\end{proof}

\section{Closable and singular relations}\label{Regsing}

Let  $T \in \bLR(\sH,\sK)$ be a linear operator or relation.
The relation $T$ is said to be \textit{closable} if the closure $T^{**}$ of $T$
is the graph of a linear operator; i.e., $\mul T^{**}=\{0\}$.  Since $\mul T^{**}=(\dom T^*)^\perp$,
it is clear that $T$ is closable if and only if $\dom T^*$ is dense in $\sK$.
The relation $T$ is said to be \textit{singular} if $T^{**}=\dom T^{**} \times \ran T^{**}$, in which
case $\dom T^{**}$ and $\ran T^{**}$ are closed. Clearly $T$ is singular if and only if
$\dom T^{**} \subset \ker T^{**}$ or $\ran T^{**} \subset \mul T^{**}$. Equivalently,
one sees that $T^*$ is singular precisely when $T^*=\dom T^* \times \mul T^*$, which is equivalent to
 $\dom T^* \subset \ker T^*$
or $\ran T^* \subset \mul T^*$.
 In particular, $T$ is singular if and only if  $T^*$ is singular.
These characterizations can be found e.g. in \cite{BHS}, \cite{HSS2018}.

\medskip

Let $T \in \bLR(\sH,\sK)$ and let $R \in \bB(\sK)$ be a nonnegative contraction.
The interest is in properties of the product
\[
 RT=\big\{ \{f, Rf'\} : \, \{f,f'\} \in T\big\},
\]
so that $RT \in \bLR(\sH,\sK)$  with $\dom RT=\dom T$. Recall the general fact that
\[
\mul RT =R \,\mul T.
\]
In particular, $RT$ is an operator if and only if
\begin{equation}\label{ber1+}
 \mul T \subset \ker R.
\end{equation}
Since $R \in \bB(\sK)$ one has $(RT)^*=T^*R$ (cf.  \cite{BHS})
and this leads to the inclusions
\begin{equation}\label{ber2}
 RT^{**} \subset (RT)^{**} \quad \mbox{and} \quad R \,\mul T^{**} \subset \mul (RT)^{**}.
\end{equation}
To proceed, some useful observations about the relation $(RT)^*=T^*R$
 are needed.
Define the linear subset $\cD \subset \cran R$ by
\begin{equation}\label{dense0}
 \cD=  \big\{k\in\cran R:\, R k \in \dom T^* \big\}.
\end{equation}
Then it is clear from the decomposition $\sK=\ker R\oplus \cran R$
that
\begin{equation}\label{dense}
 \dom T^*R=\big\{k\in\sK:\, R k \in \dom T^* \big\}=\ker R \oplus \cD.
  \end{equation}
It follows from \eqref{dense} and the definition in \eqref{dense0}, respectively, that
\begin{equation}\label{dense1}
 R( \dom T^*R) =R\cD= \ran R\cap \dom T^*.
\end{equation}

\medskip

The next two lemmas give criteria for the relation $RT$
to be closable or singular, respectively;
see \cite{HSS2018} for the case where $R$ is an orthogonal projection.
First the characterization of the closable case will be considered.

\begin{lemma}\label{LebtypelemA}
Let $T \in \bLR(\sH,\sK)$ and
let $R \in\mathbf{B}(\sK)$ be a nonnegative contraction.
Then $R T$ is closable if and only if
\begin{equation}\label{RTreg}
  \clos \{k\in\cran R:\, R k \in \dom T^* \}=\cran R.
\end{equation}
Furthermore, if $RT$ is closable, then
\begin{equation}\label{RTregb}
  \clos (\ran R\cap \dom T^*)=\cran R,
\end{equation}
and, in particular,
\begin{equation}\label{RTregb-}
  \ran R\subset\cdom T^* \quad \mbox{or, equivalently,}
  \quad \mul T^{**}\subset \ker R.
\end{equation}
If $\ran R$ is closed, then the conditions \eqref{RTreg}
and \eqref{RTregb} are equivalent.
Moreover, if $R \in\mathbf{B}(\sK)$ is invertible,
then $RT$ is closable if and only if $T$ is closable.
\end{lemma}

\begin{proof}
Recall that $R T$ is closable if and only if its adjoint $(RT)^*$ is densely defined.
 Thus it follows from $\dom (RT)^*=\dom T^*R$
 and \eqref{dense} that $RT$ is closable if and only if
$\cD$ is dense in $\cran R$, i.e., if and only if \eqref{RTreg} is satisfied.

Now assume that $RT$ is closable, i.e.,  \eqref{RTreg} holds.
Then $\cD$ is dense in $\cran R$. As a consequence,
also $R\cD$ is dense in $\cran R$. Thanks to \eqref{dense1}
one sees that \eqref{RTregb} holds.

The assertion $\mul T^{**}\subset \ker R$ in \eqref{RTregb-}
follows directly from \eqref{ber2}.
It is clearly equivalent to  $\ran R\subset\cdom T^*$.
Both assertions can also be seen as consequences of
the identity \eqref{RTregb}.

As to the last assertions, it suffices to show that \eqref{RTregb} implies \eqref{RTreg}
if $\ran R$ is closed. In this case $R$ maps $\cran R$ bijectively onto itself
and it follows from \eqref{dense1}  that $\cD=R^{-1} (\ran R \cap \dom T^*)$.
Thus if $\ran R \cap \dom T^*$ is dense in $\cran R$ then $\cD$ is dense $\cran R$.
 Therefore, \eqref{RTregb} implies \eqref{RTreg}.
 \end{proof}

Note that in the special case when $R$ is an orthogonal projection
closability of $RT$ was characterized in \cite[Lemma 2.5, Lemma 3.4]{HSS2018}
via the condition \eqref{RTregb}.

\begin{corollary}\label{Lebtypelem1}
With $T$ and $R$ as in Lemma {\rm \ref{LebtypelemA}} the following statements hold:
\begin{enumerate} \def\labelenumi{\rm(\alph{enumi})}

\item If $R T$ is closable and $\mul T^{**} \cap \ker R=\{0\}$,
then $T$ is closable.

\item If $\dom T^*$ is closed, then $R T$ is closable if and only if $\ran R \subset \dom T^*$.
In this case  $(RT)^{**}\in\mathbf{B}(\cdom T,\sK)$.
\end{enumerate}
\end{corollary}

\begin{proof}
(a) If $RT$ be closable then $\mul T^{**} \subset \ker R$ by Lemma \ref{LebtypelemA}.
An equivalent statement is
 $\mul T^{**}=\mul T^{**} \cap \ker R$. Thus (a) is clear.

(b) Assume that $\dom T^*$ is closed. If $R T$ is closable, then $\ran R \subset \dom T^*$
by Lemma \ref{LebtypelemA}.
Conversely, if $\ran R \subset \dom T^*$ then $\dom T^*R=\sK$ and $R T$ is closable.
Since $\dom T^*R=\sK$ and $(RT)^{**}=(T^*R)^*$,
the domain of $(RT)^{**}$ is closed (see \cite{BHS}) and hence equal to $\cdom T$.
Thus, $(RT)^{**}\in\mathbf{B}(\cdom T,\sK)$ by the closed graph theorem.
\end{proof}

Next the characterization of the singular case will be considered.

\begin{lemma}\label{LebtypelemB}
Let $T \in \bLR(\sH,\sK)$ and
let $R \in\mathbf{B}(\sK)$ be a nonnegative contraction.
Then $R T$ is singular if and only if
\begin{equation}\label{KKsingQ}
    \ran R \cap \dom T^* \subset \ker T^*.
\end{equation}
If $T \in \bLR(\sH, \sK)$ has a dense range,
then $RT$ is singular if and only if
\[
\ran R \cap \dom T^* =\{0\}.
\]
\end{lemma}

\begin{proof}
Recall that $RT$ is singular if and only if the adjoint
$(R T)^*=T^*R$ is singular or, equivalently,
 \begin{equation}\label{K3}
\dom T^*R \subset \ker T^*R.
\end{equation}
 Since by \eqref{dense1} one has $R(\dom T^*R)=R\cD$,
the condition \eqref{K3} holds if and only if $R\cD\subset \ker T^*$.
By \eqref{dense1} this is equivalent to
 \eqref{KKsingQ}.
\end{proof}

\begin{remark}
The characterization of closability in Lemma  \ref{LebtypelemA}
has an alternative formulation.
If the relation $R T$ is closable then $\dom T^*R$ is dense, which
implies  $\mul T^{**} \subset \ker R$ (cf. \eqref{ber2}), and then 
\begin{equation*}%\label{ddense}
\begin{split}
 \dom T^*R& =\{k\in\sK:\, R k \in \dom T^* \} \\
 &=\mul T^{**} \oplus \{k\in \cdom T^*:\, R k \in \dom T^* \},
\end{split}
\end{equation*}
where now the orthogonal decomposition $\sK=\cdom T^* \oplus \mul T^{**}$ is used.
It is easily seen that the closability of $RT$ is equivalent to
\begin{equation*}%\label{KLregQ+}
 \left\{ \begin{array}{l} \mul T^{**} \subset \ker R, \\
                                   \clos \{k\in\cdom T^*:\, R k \in \dom T^* \}=\cdom T^*.
           \end{array}
\right.
\end{equation*}
\end{remark}

\section{Pseudo-orthogonal Lebesgue type decompositions}\label{lebestype}

In this section the general notion of a pseudo-orthogonal
Lebesgue type decomposition for linear operators
or relations is developed. In \cite{HSS2018} the Lebesgue type decompositions of a
linear relation $T$ were always orthogonal. The new notion
allows a nontrivial intersection of the components;
cf.  Theorem \ref{overl}.

\begin{definition}\label{nonorth}
Let the relations $T$, $T_1$, and $T_2$ belong to $\bLR(\sH,\sK)$.
Then the sum decomposition
\begin{equation}\label{ltd}
  T=T_1+T_2 \quad \mbox{with} \quad \dom T=\dom T_1=\dom T_2,
\end{equation}
is called a pseudo-orthogonal Lebesgue type decomposition
if it is a pseudo-ortho\-gonal decomposition as in Definition {\rm \ref{nonorthsum}},
such that $T_1$ is closable and $T_2$ is singular.
\end{definition}

The following characterization of pseudo-orthogonal Lebesgue type decompositions is
a straightforward consequence of Theorem \ref{prelem0sum},
Lemma  \ref{LebtypelemA}, and Lemma \ref{LebtypelemB}.
Note that now the condition \eqref{sum3} is automatically satisfied.

\begin{theorem}\label{prelem0}
Let $T \in \bLR(\sH,\sK)$ be a linear relation.
Assume that  $K\in\mathbf{B}(\sK)$ is a nonnegative contraction
which satisfies
 \begin{equation}\label{comsum1a}
\clos \big\{k\in\cran (I-K):\, (I-K) k \in \dom T^* \big\}=\cran (I-K),
\end{equation}
 \begin{equation}\label{comsum1b}
 \ran K \cap \dom T^* \subset \ker T^*,
\end{equation}
and define
\begin{equation}\label{comsum2}
 T_1=(I-K)T \quad \mbox{and} \quad T_2=KT.
\end{equation}
Then the sum $T=T_1+T_2$ as in \eqref{ltd}
is a pseudo-orthogonal  Lebesgue type decomposition of $T$,
connected with  the pair $I-K$ and $K$ in the sense of Definition {\rm \ref{nonorth}}.

Conversely, let the sum $T=T_1+T_2$ in \eqref{ltd} be a pseudo-orthogonal Lebesgue type decomposition
of $T \in \bLR(\sH,\sK)$ in the sense of Definition {\rm \ref{nonorth}}. Then
there exists a nonnegative contraction $K \in \bB(\sK)$
such that \eqref{comsum1a}, \eqref{comsum1b},
and \eqref{comsum2} are satisfied.
\end{theorem}

\begin{proof}
Let $K \in \bB(\sK)$ be a nonnegative contraction and assume that \eqref{comsum1a}
and \eqref{comsum1b} hold. Then $T_1=(I-K)T$ is a closable operator and $T_2=KT$
is a singular relation by  Lemma  \ref{LebtypelemA} and Lemma  \ref{LebtypelemB}.
Hence $\mul T_1=\{0\}$ so that \eqref{sum3}
is satisfied. By Theorem \ref{prelem0sum} $T=T_1+T_2$ is a pseudo-orthogonal decomposition,
which is a pseudo-orthogonal Lebesgue type decomposition according to Definition \ref{nonorth}.

Conversely, let $T=T_1+T_2$ be a pseudo-orthogonal Lebesgue type decomposition.
Hence, by definition it is a pseudo-orthogonal decomposition,
where $T_1$ is closable and $T_2$ is singular.
According to Theorem \ref{prelem0sum} there
exists a nonnegative contraction $K \in \bB(\sK)$ for which
the identities in \eqref{sum3} (trivially, since $\mul T_1=\{0\}$) and \eqref{comsum2} hold.
In fact, by Lemma \ref{LebtypelemA} and Lemma  \ref{LebtypelemB} the assertions
in \eqref{comsum1a} and \eqref{comsum1b}
follow.
\end{proof}

The sum decomposition \eqref{ltd} in Definition \ref{nonorth}
is said to be an \textit{orthogonal Lebesgue type decomposition}  if
it is an orthogonal decomposition as in Definition \ref{orthsum},
such that $T_1$ is closable and $T_2$ singular.
Hence, the following characterization of orthogonal Lebesgue type decompositions is
a direct consequence of Theorem \ref{prelem0}, Lemma  \ref{LebtypelemA},
and  Lemma  \ref{LebtypelemB};  cf. \cite{HSS2018}.

\begin{corollary} \label{prelem1}
Let $T \in \bLR(\sH,\sK)$ be a linear relation.
 Assume that  $P \in\mathbf{B}(\sK)$ is an orthogonal projection
 which satisfies
\begin{equation}\label{comsum1ap}
  \clos ( \ker P \cap \dom T^*) =\ker P,
\end{equation}
\begin{equation}\label{comsum1bp}
 \ran P \cap \dom T^* \subset \ker T^*,
\end{equation}
and define
 \begin{equation}\label{comsum2p}
 T_1=(I-P)T \quad \mbox{and} \quad T_2=PT.
\end{equation}
Then  the sum  $T=T_1+T_2$ as in \eqref{ltd}
is an orthogonal Lebesgue type decomposition of $T$,
connected with the pair  $I-P$ and $P$.

Conversely, let the sum $T=T_1+T_2$ in \eqref{ltd} be an orthogonal  Lebesgue type decomposition
of $T \in \bLR(\sH,\sK)$.
Then there exists an orthogonal projection $P \in \bB(\sK)$
such that \eqref{comsum1ap}, \eqref{comsum1bp},
and \eqref{comsum2p} are satisfied.
\end{corollary}

Let $P_0$ stand for   the orthogonal projection onto $\mul T^{**}$.
Then it is clear that the conditions \eqref{comsum1ap} and \eqref{comsum1bp}
in Corollary \ref{prelem1} are  satisfied and it follows that
 \begin{equation}\label{lebe}
 T=T_{\rm reg} + T_{\rm sing}, \quad T_{\rm reg}=(I-P_0)T, \quad T_{\rm sing}=P_0 T,
\end{equation}
is an orthogonal Lebesgue type decomposition of $T$.  Here the \textit{regular part}
$T_{\rm reg} $ is closable and the \textit{singular part}   $T_{\rm sing} $ is singular.
The decomposition in \eqref{lebe} is called the \textit{Lebesgue decomposition} of $T$; cf.
\cite{HSSS2007, J80, KO, Ota84, Ota87}.
The Lebesgue  decomposition in \eqref{lebe} shows the existence of
Lebesgue type decompositions of $T$.
Note that $T_{\rm reg}$ is bounded if and only if $\dom T^*$ is closed; cf. \cite{HSS2018}.

\medskip

Among all Lebesgue type decomposition of a linear relation $T \in \bLR(\sH,\sK)$
the Lebesgue decomposition in \eqref{lebe} is distinguished by the maximality
property of its regular part $T_{\rm reg}$.
Recall that for linear relations $S_1$ and $S_2$ from $\sH$ to $\sK$
one says that $S_1$ is \textit{dominated} (\textit{contractively dominated}) by $S_2$,
in notation $S_1 \prec S_2$ ($S_1 \prec_c S_2$), if $CS_2 \subset S_1$
for some bounded (contractive) operator $C \in \bB(\sK)$.
When $S_1$ and $S_2$ are operators this is equivalent to
$\dom S_2 \subset \dom S_1$ and
$\|S_1 h\| \leq c \|S_2 h\|$ for all $h \in \dom S_2$ for some $0< c$ ($0<c \leq 1$); cf. \cite{HS2015}
and \cite[Definition 8.1, Lemma 8.2]{HSS2018}.
The next result is a strengthening of the maximality property established earlier
for orthogonal Lebesgue type decompositions in \cite{HSS2018}
to the wider setting of pseudo-orthogonal Lebesgue type decompositions of $T$.

\begin{theorem}\label{dom}
Let $T \in \bLR(\sH,\sK)$ and let $T=T_1+T_2$ be
a pseudo-orthogonal Lebesgue type decomposition of $T$.
Then
\[
 T_1 \prec_c T_{\rm reg},
\]
that is, the regular part $T_{\rm reg}$ of the Lebesgue decomposition
is the maximal closable part of $T$, in the sense of domination,
among all pseudo-orthogonal Lebesgue type decompositions of $T$.
\end{theorem}

\begin{proof}
In $T=T_1+T_2$ one has $T_1=(I-K)T$ for a nonnegative contraction
and note that $I-K$ is also a nonnegative contraction.
Hence one concludes $T_1 \prec_c T$. This domination
is preserved by their regular parts, see \cite[Theorem 8.3]{HSS2018}.
Since $T_1$ is closable, it is equal to its regular part and
it follows that $T_1 \prec_c T_{\rm reg}$.
\end{proof}

For a further consideration of Lebesgue type decompositions
the class of contractions in $\bB(\sK)$ will now be restricted to contractions of the form
\begin{equation}\label{konk}
 K=\begin{pmatrix} I & 0 \\ 0 & G \end{pmatrix}:
 \begin{pmatrix} \mul T^{**} \\ \cdom T^* \end{pmatrix} \to
 \begin{pmatrix} \mul T^{**} \\ \cdom T^* \end{pmatrix},
\end{equation}
where $G \in \bB(\cdom T^*)$ is a nonnegative contraction.
It follows from Theorem \ref{prelem0} that $K$ in \eqref{konk}
satisfies \eqref{comsum1a} if and only if
\begin{equation}\label{twee}
   \clos \big\{k\in\cran (I-G):\, (I-G) k \in \dom T^* \big\}=\cran (I-G),
\end{equation}
and $K$ satisfies \eqref{comsum1b} if and only if
\begin{equation}\label{een}
\ran G  \cap \dom T^* \subset \ker T^*.
\end{equation}
Conversely, any $G \in \bB(\cdom T^*)$ with the properties \eqref{twee} and \eqref{een}
gives via \eqref{konk} a nonnegative
contraction $K \in \bB(\sK)$ as in Theorem \ref{prelem0}.
The case $G=0$ corresponds to $K=P_0$, the orthogonal projection onto $\mul T^{**}$,
and gives the Lebesgue decomposition, while any orthogonal Lebesgue type
decomposition corresponds via \eqref{konk} to an orthogonal projection $G$
that satisfies \eqref{twee} and \eqref{een}.

\medskip

Now assume that $\dom T^*$ is not closed.
Let $\sL \subset \cdom T^* \setminus \dom T^*$
be a closed linear subspace of $\cdom T^*$ and decompose this space accordingly:
\[
\cdom T^*=(\cdom T^* \ominus \sL) \oplus \sL.
\]
This decomposition will be used in  the lemma below.
As to the existence of such subspaces $\sL$, recall
that $\dom T^*$ is an operator range space.
Hence one has $\dim (\cdom T^* \setminus \dom T^*)=\infty$;
see \cite[Corollary to Theorem 2.3]{FW}.
Therefore one may choose for any $n \in \dN$ an $n$-dimensional linear subspace
$\sL  \subset \cdom T^* \setminus \dom T^*$.
The following lemmas describe special classes of
nonnegative contractions $K \in \bB(\sK)$ that illustrate
several features discussed earlier.

\begin{lemma}\label{kontt}
Let $T \in \bLR(\sH,\sK)$ and  assume that $\dom T^*$ is not closed.
Let $\sL$ be a nontrivial closed linear subspace of $\cdom T^* \setminus \dom T^*$.
Let $H \in\mathbf{B}(\sL)$ be a nonnegative contraction,
then the operator $G$ defined by
\begin{equation}\label{driek}
 G=\begin{pmatrix} 0& 0 \\  0&  H \end{pmatrix}:
 \begin{pmatrix}  \cdom T^{*} \ominus \sL\\ \sL \end{pmatrix}
 \to
 \begin{pmatrix} \cdom T^{*} \ominus \sL \\ \sL \end{pmatrix},
\end{equation}
belongs to $\bB(\cdom T^*)$ and satisfies the condition \eqref{een}.
Assume in addition that $(I-H)^{-1} \in \bB(\sL)$, then
the operator $G$ in \eqref{driek} satisfies
the condition \eqref{twee}.
Hence $K$ in \eqref{konk} satisfies the conditions \eqref{comsum1ap} and
\eqref{comsum1bp}. Consequently, the sum $T=T_1+T_2$ with \eqref{comsum2}
is a pseudo-orthogonal Lebesgue type decomposition of $T$.
\end{lemma}

\begin{proof}
Let $G$ be as in \eqref{driek}. Then $\ran G \subset \sL$, so that
$\ran G \cap \dom T^*=\{0\}$ by the definition of $\sL$.
Hence, the condition \eqref{een} is automatically satisfied.
Furthermore, one sees from the condition $(I-H)^{-1} \in \bB(\sL)$
that $I-G \in \bB(\cdom T^*)$ is invertible. Thus the linear space
$(I-G)^{-1} \dom T^*$ is dense in $\cdom T^*$.  Therefore \eqref{twee} is satisfied
if $(I-H)^{-1} \in \bB(\sL)$.
\end{proof}

The next lemma goes back to \cite{HSS2018}.

\begin{lemma}\label{kondt}
Let $T \in \bLR(\sH,\sK)$ and  assume that $\dom T^*$ is not closed.
Let $\sL$ be a finite-dimensional linear subspace of $\cdom T^* \setminus \dom T^*$.
Let $H \in \bB(\sL)$ be an orthogonal projection.
Then the operator $K \in \bB(\sK)$, defined by \eqref{konk} and \eqref{driek},
 is an orthogonal  projection $K=P$
 which satisfies the conditions \eqref{comsum1ap} and
\eqref{comsum1bp}.  Consequently,  the sum $T=T_1+T_2$ with \eqref{comsum2p}
is an orthogonal Lebesgue type decomposition of $T$.
\end{lemma}

Lemma \ref{kontt}  and Lemma \ref{kondt} answer questions about the existence
of Lebesgue type decompositions, different from the Lebesgue decomposition:
 when $\dom T^*$ is not closed there are infinitely many different Lebesgue type
decompositions of $T$, both pseudo-orthogonal and orthogonal.
A necessary and sufficient condition for the uniqueness of the Lebesgue decomposition
among all pseudo-orthogonal Lebesgue type decompositions
(thus including the orthogonal ones) is given in the next theorem.

\begin{theorem} \label{uniq}
Let $T \in \bLR(\sH,\sK)$,  then the following statements are equivalent:
\begin{enumerate}\def\labelenumi{\rm(\roman{enumi})}
\item  the Lebesgue decomposition of $T$
is the only pseudo-orthogonal  Lebesgue type decomposition of $T$;
\item  $\dom T^*$ is closed.
\end{enumerate}
\end{theorem}

\begin{proof}
(i) $\Rightarrow$ (ii)
Assume that  $\dom T^*$ is not closed.
According to Lemma  \ref{kontt} and Lemma \ref{kondt}
there exist Lebesgue type decompositions of $T$, which are
pseudo-orthogonal or orthogonal,
which are different from the Lebesgue decomposition.
This contradiction shows (ii).

(ii) $\Rightarrow$ (i)
Assume that $\dom T^*$ is closed.
Let $T=(I-K)T+KT$ have a Lebesgue type decomposition
\eqref{comsum2}, where $K$ is a nonnegative contraction; cf. Corollary \ref{Lebtypelem1}.
Then with the convention \eqref{konk} one has $\ran G \subset \dom T^*$ which, combined
 with \eqref{een}, leads to $\ran G \subset \ker T^*$ or, equivalently,  $\ran T \subset \ker G$.
It follows from \eqref{konk}  that
\[
 K-P_0=\begin{pmatrix} 0 & 0 \\ 0 & G \end{pmatrix}.
\]
Therefore the following identities
\[
 (I-K)T=(I-P_0)T=T_{\rm reg} \quad \mbox{and} \quad \quad KT=P_0T=T_{\rm sing}
\]
are clear.
Thus the pseudo-orthogonal decomposition of $T$ corresponding to $K$
coincides with the Lebesgue decomposition.
\end{proof}

Theorem \ref{uniq} is a strengthening of the corresponding result in \cite{HSS2018}
from the case of orthogonal Lebesgue type decompositions to the case of
pseudo-orthogonal Lebesgue type decompositions.
The uniqueness condition in (ii) is equivalent to the condition that the operator
$T_{\rm reg}$ is bounded, see  \cite{HSS2018}.
To see this equivalence, recall that  $\dom T^{*}$ is closed if and only if
$\dom T^{**}$ is closed, while
\[
\dom T^{**}=\dom (T^{**})_{\rm reg}=\dom (T_{\rm reg})^{**}.
\]
The original statement of such a uniqueness result
in the setting of pairs of nonnegative bounded operators goes back to
Ando \cite{An}. In \cite{ST} there is an extensive treatment
of the uniqueness question
in the context of forms, including a list of the relevant literature.

\medskip

Finally, it should be observed that  Lemma \ref{kontt}  provides some
concrete examples for nontrivial intersection of the components in a Lebesgue type
decomposition.

\begin{corollary}\label{kont}
Assume the conditions in Lemma {\rm \ref{kontt}}
and let $T=T_1+T_2$ be the corresponding Lebesgue type decomposition.
Then the following statements hold.
\begin{enumerate} \def\labelenumi{\rm(\alph{enumi})}
\item
The intersection of $\ran T_1$ and  $\ran T_2$ satisfies
\begin{equation}\label{aaaa}
\ran T_1 \cap \ran T_2 = \{0\}_{\sK\ominus\sL}\oplus  H(I-H) P_\sL \sM,
\end{equation}
where $\sM$ is given by \eqref{ell} and $P_\sL$ is the orthogonal projection onto $\sL$.
\item
If $\cran T=\sK$, then the intersection of $\cran T_1$ and $\cran T_2$ satisfies
\begin{equation}\label{aaaab}
\cran T_1\cap \cran T_2=\{0\}_{\sK\ominus\sL}\oplus \cran H.
\end{equation}
Consequently,  if $H \neq 0$ then $\cran T_1\cap \cran T_2\neq \{0\}$.
\end{enumerate}
\end{corollary}

\begin{proof}
First observe with the matrix representations in \eqref{konk} and \eqref{driek} that
\begin{equation}\label{gh}
 (I-K)K=\begin{pmatrix} 0 & 0\\0& (I-G)G \end{pmatrix} \quad \mbox{and} \quad
  (I-G)G=\begin{pmatrix} 0 & 0\\0& (I-H)H \end{pmatrix}.
\end{equation}

(a) The description \eqref{aaaa} is obtained directly
from Theorem \ref{overl} by using the block formulas in \eqref{gh}.

(b) By assumption $I-H$ is surjective and hence $\cran (I-H)H=\cran H$.
Since $T$ is minimal, the statement in \eqref{aaaab}
follows from Lemma \ref{lap} again by means of \eqref{gh}.
\end{proof}

\section{Pairs of bounded linear operators}\label{OpRan}

Let $\Phi\in\bB(\sE,\sH)$ and $\Psi\in\bB(\sE,\sK)$ be bounded linear operators.
With these operators one associates the linear relation $L(\Phi,\Psi) \in \bLR(\sH, \sK)$,
defined by
\begin{equation}\label{Tabegin}
 L(\Phi, \Psi)=\big\{\,\{\Phi \eta, \Psi \eta \}:\,  \eta \in\sE\,\big\},
\end{equation}
so that $L(\Phi, \Psi)$ is an operator range relation in the sense of \cite{BHS}, \cite{HS2022a}.
It follows directly from the definition of $L(\Phi, \Psi)$ that its domain and range are given by
\begin{equation}\label{Tabeginn}
\dom L(\Phi, \Psi)=\ran \Phi \quad \mbox{and} \quad \ran L(\Phi, \Psi)=\ran \Psi,
\end{equation}
while its kernel and multivalued part are given by
\begin{equation}\label{Tabb}
 \ker L(\Phi, \Psi)=\Phi(\ker \Psi), \quad \mul L(\Phi, \Psi)= \Psi(\ker \Phi).
\end{equation}
This section gives a brief overview of the decompositions $\Psi=\Psi_1+\Psi_2$
with respect to $\Phi$, with bounded operators $\Psi_1$ and $\Psi_2$,
in the present context of a pseudo-orthogonal decomposition of the space $\sK$,
allowing interaction between the components as in Sections \ref{overlap}
and Section \ref{Regsing}.
These decompositions of $\Psi$ with respect to $\Phi$
will be obtained via the corresponding decompositions
of the corresponding linear relation $L(\Phi, \Psi)$;
for the orthogonal case, see \cite{HS2022a}.

\medskip

The present interest is in sums  $\Psi=\Psi_{1}+\Psi_{2}$
 and their interplay with the corresponding linear relations $L(\Phi, \Psi_1+\Psi_2)$.

\begin{lemma}\label{uiteind}
Let $\Phi\in\bB(\sE,\sH)$, $\Psi\in\bB(\sE,\sK)$, and assume that
$\Psi=\Psi_{1}+\Psi_{2}$ where $\Psi_1, \Psi_2 \in \bB(\sE,\sK)$. Then
there is domain equality
\[
\dom L(\Phi, \Psi)=\dom L(\Phi, \Psi_1) =\dom L(\Phi, \Psi_2),
\]
and inclusion of the relations
\begin{equation}\label{vaasneww-}
L(\Phi, \Psi) \subset L(\Phi, \Psi_{1})+L(\Phi, \Psi_{2}).
\end{equation}
Moreover, there is equality in \eqref{vaasneww-}:
\begin{equation}\label{telaat}
 L(\Phi, \Psi)=L(\Phi, \Psi_1) +L(\Phi, \Psi_2)
\end{equation}
if and only if
\begin{equation}\label{mmulli}
\Psi(\ker \Phi)= \Psi_1(\ker \Phi)+ \Psi_2(\ker \Phi).
\end{equation}
The sum in \eqref{telaat} is strict (i.e., the sum in \eqref{mmulli} is direct) if and only if
\begin{equation}\label{mmulli+}
\Psi_1(\ker \Phi) \cap \Psi_2(\ker \Phi)=\{0\}.
\end{equation}
\end{lemma}

\begin{proof}
From the definition of the relation $L(\Phi, \Psi)$  in \eqref{Tabegin}  it is clear that
\[
\dom L(\Phi,\Psi)=\dom L(\Phi, \Psi_1)=\dom L(\Phi, \Psi_2)=\ran \Phi,
\]
see  \eqref{Tabeginn}.
Furthermore, it follows from the definition of the sum in \eqref{aaa0} that \eqref{vaasneww-} holds.
Now recall from Lemma \ref{opnieuw} that there is equality in \eqref{vaasneww-}
if and only if
\[
\mul L(\Phi, \Psi)=\mul L(\Phi,\Psi_1) +\mul L(\Phi,\Psi_2),
\]
which is clearly equivalent to \eqref{mmulli}; cf. \eqref{Tabb}.
\end{proof}

\begin{remark}\label{uiteind2}
Let $\sK$ have a pseudo-orthogonal decomposition
$\sK=\sX+\sY$ with associated nonnegative contractions $X$ and $Y$ such that $X+Y=I$.
Then
by Definition \ref{nonorthsum} the decomposition \eqref{telaat} of the relation $L(\Phi, \Psi)$
is pseudo-orthogonal if and only if
\begin{itemize}\def\labelenumi{\rm(\alph{enumi})}
\item[(a)] $\ran \Psi_1 \subset \sX$ and $\ran \Psi_2 \subset \sY$;
\item[(b)]   for each $\eta \in \sE$ there exist elements $\eta', \eta'' \in \sE$
with $\Phi\eta=\Phi \eta'=\Phi \eta''$ and
$\Psi \eta =\Psi_1 \eta'+\Psi_2 \eta''$ for which
 \[
 \|\Psi \eta \|_\sK^2=\|\Psi_1 \eta' \|_\sX^2 + \|\Psi_2 \eta'' \|_\sY^2.
\]
\end{itemize}
  Note that (b) implies $\eta-\eta', \eta-\eta'' \in \ker \Phi$
and $\Psi_1\eta'+\Psi_2 \eta''=\Psi \eta=\Psi_1 \eta+\Psi_2 \eta$, so that
\begin{equation}\label{neu0}
\Psi_1 (\eta' -\eta)= \Psi_2(\eta-\eta'').
 \end{equation}
By Lemma \ref{nonorthsum1} the sum in \eqref{telaat} is strict,
thus one has  \eqref{mmulli+}. Therefore \eqref{neu0} gives that
$\Psi_1 \eta'=\Psi_1 \eta$ and $\Psi_2 \eta''=\Psi_2 \eta$.
Hence (b) implies that
\begin{equation}\label{neu}
\|\Psi \eta \|_\sK^2=\|\Psi_1 \eta \|_\sX^2 + \|\Psi_2 \eta \|_\sY^2, \quad \eta \in \sE.
\end{equation}
Note that if \eqref{neu} is satisfied, then (b) holds automatically.
Thus the conditions (b) and \eqref{neu} are equivalent.
\end{remark}

\begin{definition}\label{uiteind3}
 Let $\Phi$, $\Psi$, $\Psi_1$, and $\Psi_2$ be bounded linear operators in $\bB(\sE,\sH)$
and assume that $\Psi$ has the decomposition
\begin{equation}\label{bbb}
\Psi=\Psi_{1}+\Psi_{2} \quad \mbox{with} \quad \Psi(\ker \Phi)=\Psi_1(\ker \Phi) + \Psi_2(\ker \Phi),
 \quad \mbox{direct sum}.
\end{equation}
Let $\sK$ have a \textit{pseudo-orthogonal decomposition}
$\sK=\sX+\sY$ with associated nonnegative contractions $X$ and $Y$ such that $X+Y=I$.
Then the decomposition \eqref{bbb} of $\Psi$ with respect to $\Phi$
is called pseudo-orthogonal if $\Phi$, $\Psi$, $\Psi_1$, and $\Psi_2$
satisfy the conditions
\begin{itemize}
\item[(a)]  $\ran \Psi_1 \subset \sX$ and $\ran \Psi_2 \subset \sY$;
\item[(b)]  $\|\Psi \eta \|_\sK^2=\|\Psi_1 \eta \|_\sX^2 + \|\Psi_2 \eta \|_\sY^2$
for all $\eta \in \sE$.
\end{itemize}
\end{definition}

It is clear from Remark \ref{uiteind2}  that the
decomposition \eqref{bbb} of $\Psi$ with respect to $\Phi$
is pseudo-orthogonal if and only if the corresponding operator range relation $L(\Phi, \Psi)$
in \eqref{Tabegin}  is pseudo-orthogonal.
 The pseudo-orthogonal decompositions of $\Psi$ with respect to $\Phi$
in Definition \ref{uiteind3} can now be characterized by means of nonnegative
contractions in $\bB(\sK)$.

\begin{theorem}\label{mainab0}
Let $\Phi\in\bB(\sE,\sH)$ and $\Psi\in\bB(\sE,\sK)$.
Assume that $K \in \bB(\sK)$ is a nonnegative contraction
which satisfies
\begin{equation}\label{comsum1aw0}
\Psi (\ker \Phi)=(I-K) \Psi(\ker \Phi) + K \Psi(\ker \Phi), \quad \mbox{direct sum},
\end{equation}
and define
\begin{equation}\label{comsum1cw0}
\Psi_1=(I-K)\Psi \quad \mbox{and}  \quad \Psi_2=K\Psi.
\end{equation}
Then the sum $\Psi=\Psi_1+\Psi_2$ as in \eqref{bbb}
is a pseudo-orthogonal decomposition of $\Psi$ with respect to $\Phi$,
connected with  the pair  $I-K$ and $K$ in the sense of Definition {\rm \ref{uiteind3}}.

Conversely, let $\Psi=\Psi_1+\Psi_2$  in \eqref{bbb} be a pseudo-orthogonal decomposition of $\Psi$
with respect to $\Phi$ in the sense of Definition {\rm \ref{uiteind3}}.
Then there exists a nonnegative contraction $K \in \bB(\sK)$
such that
\eqref{comsum1aw0} and  \eqref{comsum1cw0} are satisfied.
\end{theorem}

\begin{proof}
Let $\Phi \in\bB(\sE,\sH)$, $\Psi \in\bB(\sE,\sK)$, and let
$K \in \bB(\sK)$ be a nonnegative contraction.
Define the operators $\Psi_{1}, \Psi_{2} \in\bB(\sE,\sK)$ by \eqref{comsum1cw0},
so that $\Psi=\Psi_1+\Psi_2$.
 Let $\sX$ and $\sY$ be the pair of complemented operator range spaces,
contractively contained in $\sK$, associated
with the nonnegative contractions $X=I-K$ and $Y=K$.
 By definition
\[
 \ran \Psi_1 =\ran (I-K) \Psi \subset \sX
 \quad \mbox{and} \quad \ran \Psi_2 =\ran K\Psi \subset \sY,
\]
so that condition (a) in Definition \ref{uiteind3} is satisfied.
To see (b) in Definition \ref{uiteind3} observe that
 \[
\begin{split}
 \|\Psi_1 \eta \|_\sX^2 + \|\Psi_2 \eta \|_\sY^2&
 =  \|(I-K)\Psi \eta \|_\sX^2 + \|K \Psi \eta\|_\sY^2 \\
&=((I-K)\Psi \eta, \Psi \eta)_\sK+(K\Psi \eta,\Psi \eta)_\sK =\|\Psi \eta\|_\sK^2,
\end{split}
\]
so that condition (b) in Remark \ref{uiteind2} is satisfied.
Hence, $\Psi=\Psi_1+\Psi_2$ is a  pseudo-orthogonal decomposition
of $\Psi$ with respect to $\Phi$ in the sense of Definition \ref{uiteind3}.

Conversely, assume that $\Psi=\Psi_1+\Psi_2$ is a pseudo-orthogonal decomposition
with respect to $\Phi$ as in Definition \ref{uiteind3}.
Then $L(\Phi, \Psi)$ in \eqref{Tabegin} has a pseudo-orthogonal
decomposition of the form
\[
L(\Phi, \Psi)=L(\Phi, \Psi_{1})+L(\Phi, \Psi_{2}).
\]
cf. Remark \ref{uiteind2}.
Therefore, by Theorem \ref{prelem0sum} there exists a
nonnegative contraction $K \in \bB(\sK)$ gives
\[
L(\Phi, \Psi_1)=(I-K)L(\Phi, \Psi) \quad  \mbox{and} \quad L(\Phi, \Psi_2)=KL(\Phi, \Psi),
\]
which reads
\begin{equation}\label{grijps1}
\begin{split}
L(\Phi, \Psi_1)=L(\Phi, (I-K)\Psi) \quad  \mbox{and} \quad L(\Phi, \Psi_2)=L(\Phi, K \Psi).
\end{split}
\end{equation}
To verify the identities in \eqref{comsum1cw0}
let $\eta \in \sE$. Then due to  \eqref{grijps1} there exist $\eta', \eta'' \in \sE$ such that
$\Phi \eta=\Phi \eta'=\Phi \eta''$,
while $(I-K)\Psi \eta=\Psi_1 \eta'$ and $K\Psi \eta=\Psi_2 \eta''$.
From $\Psi_1\eta'+\Psi_2 \eta''=\Psi \eta =\Psi_1 \eta+\Psi_2 \eta$
it follows that $\Psi_1 \eta'=\Psi_1 \eta$ and $\Psi_2 \eta''=\Psi_2 \eta$;
cf. \eqref{mmulli+} and Remark \ref{uiteind2}.
Therefore, the identities in \eqref{comsum1cw0} hold.
\end{proof}

Now consider the case of operator range relations
$L(\Phi, \Psi)$ that are closable or singular.
Recall that $L(\Phi, \Psi)$ is closable if and only if $\mul L(\Phi, \Psi)^{**}=\{0\}$ or, equivalently,
for every sequence $\eta_{n} \in \sE$ one has
\begin{equation}\label{grij1}
 \Phi \eta_{n} \to 0 \quad \mbox{and} \quad
 \Psi (\eta_{n}-\eta_{m}) \to 0
 \quad \Rightarrow \quad \Psi \eta_{n} \to 0.
\end{equation}
Likewise, $L(\Phi, \Psi)$ is singular if and only if  $\ran L(\Phi, \Psi)^{**} \subset \mul L(\Phi, \Psi)^{**}$
(cf. \cite[Proposition 2.8]{HSS2018}) or, equivalently,
for every sequence $\eta_{n}$ in $\sE$  there exists a subsequence,
denoted by $\omega_{n}$, such that
\begin{equation}\label{grij2}
\Psi (\eta_{n}-\eta_{m}) \to 0 \quad \Rightarrow \quad  \Phi \omega_n \to 0.
\end{equation}
Note that $L(\Psi, \Phi)=L(\Phi, \Psi)^{-1}$ implies that
$L(\Phi, \Psi)$ is singular if and only if $L(\Psi, \Phi)$ is singular.

\begin{remark}
The characterizations in \eqref{grij1} and \eqref{grij2} of
closable and singular operator range relations
remain valid if the sequences $\varphi_n$ are taken from a dense set $\sR \subset \sE$.
To see this, observe that for any sequence $\eta_{n} \in \sE$
there exists an approximating sequence $\eta_{n}' \in \sR$, such that
\begin{equation*}%\label{CC0}
  \| \eta_{n} - \eta_{n}' \| \leq \frac{1}{n},
  \quad \mbox{in which case} \quad
  \big| \| L \eta_{n} \| - \| L \eta_{n}'\| \big|  \leq  \frac{1}{n} \|L\|,
\end{equation*}
for any  $L \in \mathbf{B}(\sE, \sL)$, where $\sL$ is a Hilbert space.
\end{remark}

The following simple observations about the adjoint relation $L(\Phi, \Psi)^*$
play a role in the rest of this section.
A direct calculation shows that
\[
 L(\Phi, \Psi)^*=\big\{\,\{k,h\}\in \sK\times\sH:\, \Psi^*k=\Phi^*h\,\big\}.
\]
Thus, by means of  the linear subspaces
\[
\sD(\Phi, \Psi)=\{k \in \sK:\, \Psi^{*}k\in \ran \Phi^{*}\},
\quad
\sR(\Phi, \Psi)=\{h \in \sH:\, \Phi^{*}h\in \ran \Psi^{*}\},
\]
the domain and the range of  $L(\Phi, \Psi)^{*}$
are given by
\begin{equation*}%\label{tabst}
 \dom L(\Phi, \Psi)^*=\sD(\Phi, \Psi), \quad  \ran L(\Phi, \Psi)^*=\sR(\Phi, \Psi),
\end{equation*}
and, likewise, the kernel and multivalued part of $L(\Phi, \Psi)^*$ are given by
\begin{equation*}%\label{tabst2}
 \ker L(\Phi, \Psi)^*=\ker \Psi^*, \quad  \mul L(\Phi, \Psi)^*=\ker \Phi^*.
\end{equation*}

\begin{definition}\label{DEFregsing}
The operator $\Psi$ is called \textit{regular with respect to} $\Phi$ if $\sD(\Phi, \Psi)$
is dense in $\sK$, which is the case if and only if the relation $L(\Phi, \Psi)$ is closable.
Likewise, the operator $\Psi$ is called \textit{singular with respect to} $\Phi$ if
$\sD(\Phi, \Psi) \subset \ker \Psi^*$ or, equivalently,  $\sR(\Phi, \Psi) \subset \ker \Phi^*$,
which is the case if and only if the relation $L(\Phi, \Psi)$ is singular.
\end{definition}

Clearly, an equivalent characterization for singularity is that
\begin{equation}\label{PhiPsiSing}
 \ran \Phi^* \cap \ran \Psi^*=\{0\},
\end{equation}
(expressing the symmetry between $\Phi$ and $\Psi$); see also \cite[Lemma 5.2]{HS2022a}.

\begin{remark}
Both notions appearing in Definition \ref{DEFregsing}
have equivalent characterizations
which resemble their measure-theoretic analogs.
In particular, $\Psi$ is regular with respect to $\Phi$
if and only if $\Psi$ is almost dominated by $\Phi$.
In this case $\Psi$ has a Radon-Nikodym derivative
with respect to $\Phi$, which is given by the closed operator
\[
R(\Phi, \Psi)=L(\Phi, \Psi)^{**},
\]
and then $\Psi$ can be written as
$\Psi=R(\Phi, \Psi)\Phi$.
Likewise,
 $\Psi$ is singular with respect to $\Phi$ (or $\Phi$ is singular with respect to $\Psi$)
precisely if for any $\varXi \in \bB(\sK)$ one has
\[
 \varXi \prec \Phi \quad \mbox{and} \quad \varXi \prec \Psi \quad \Rightarrow \quad \varXi=0.
\]
For the definitions and the arguments, see \cite[Sections~5,~6]{HS2022a}.
\end{remark}

\begin{definition}\label{optype}
Let $\Phi\in\bB(\sE,\sH)$ and $\Psi\in\bB(\sE,\sK)$.
Then $\Psi$ is said to have a pseudo-orthogonal Lebesgue type
decomposition
\begin{equation}\label{oud}
\Psi=\Psi_{1}+\Psi_{2}
\end{equation}
with respect to $\Phi$ if the sum \eqref{oud}
is a  pseudo-orthogonal decomposition of $\Psi$
with respect to $\Phi$ as in Definition {\rm \ref{uiteind3}},
where $\Psi_{1}$ is regular with respect to $\Phi$ and
$\Psi_{2}$ is singular with respect to $\Phi$.
\end{definition}

The following characterization is now straightforward.

\begin{theorem}\label{mainab}
Let $\Phi\in\bB(\sE,\sH)$ and $\Psi\in\bB(\sE,\sK)$.
Assume that $K \in \bB(\sK)$ is a nonnegative contraction
which satisfies
 \begin{equation}\label{comsum1aw}
\clos \big\{k\in\cran (I-K):\, (I-K) k \in \sD(\Phi, \Psi) \big\}=\cran (I-K),
\end{equation}
 \begin{equation}\label{comsum1bw}
 \ran K \cap \sD(\Phi, \Psi) \subset \ker \Psi^*,
\end{equation}
and define
\begin{equation}\label{comsum1cw}
  \Psi_1=(I-K)\Psi \quad \mbox{and} \quad \Psi_2=K\Psi.
\end{equation}
Then the sum $\Psi=\Psi_1+\Psi_2$ as in \eqref{oud}  is a  pseudo-orthogonal
Lebesgue type decomposition of $\Psi$ with respect to $\Phi$,
connected with the pair $I-K$ and $K$ in the sense of Definition {\rm \ref{optype}}.

Conversely, let $\Psi=\Psi_1+\Psi_2$ in \eqref{oud}  be
a pseudo-orthogonal Lebesgue type decomposition
of $\Psi$ with respect to $\Phi$ in the sense of Definition {\rm \ref{optype}}.
Then there exists a nonnegative contraction $K \in \bB(\sK)$
such that \eqref{comsum1aw}, \eqref{comsum1bw}, and \eqref{comsum1cw}
are satisfied.
\end{theorem}

To verify the theorem, recall Theorem \ref{mainab0} and apply Definition \ref{DEFregsing}
to the components $\Psi_1$ and $\Psi_2$ in \eqref{comsum1cw};
see also Theorem \ref{prelem0}, or Lemma \ref{LebtypelemA} and Lemma \ref{LebtypelemB}.
Note that the condition for the sum in \eqref{comsum1aw0},
as stated in Theorem \ref{mainab0}, is now absent since this condition
automatically follows from the condition \eqref{comsum1aw}:
one has $\Psi (\ker \Phi) \subset \ker (I-K)$; see \eqref{ber1+}, \eqref{Tabb}.
Observe, also that the singularity condition \eqref{comsum1bw} for the component $\Psi_2=K\Psi$
is equivalent to $\ran \Psi^*K \cap \ran \Phi^* = \{0\}$; cf. \eqref{PhiPsiSing}.

\medskip

Let $P_0$ be the orthogonal projection onto $\sD(\Phi, \Psi)^\perp$.
Then the pair of  operators $\Psi_{\rm reg}=(I-P_0) \Psi$ and $\Psi_{\rm sing}=P_0 \Psi$,
 gives an orthogonal Lebesgue type decomposition $\Psi=\Psi_{\rm reg}+ \Psi_{\rm sing}$
with respect to $\Phi$. It is called the \textit{Lebesgue decomposition} of $\Psi$ with respect to $\Phi$.
Note that it follows from Theorem \ref{mainab} that $(I-K)P_0=0$, so that $(I-K)(I-P_0)=I-K$.
Hence,
for the regular part $\Psi_{\rm reg}$ there is the following characterization; cf. \cite{HS2022a}.

\begin{corollary}\label{optim}
Let $\Phi\in\bB(\sE,\sH)$ and $\Psi\in\bB(\sE,\sK)$.
Let $\Psi=\Psi_1+\Psi_2$ be a pseudo-orthogonal Lebesgue type decomposition of $\Psi$
with respect to $\Phi$, then
\[
 \|\Psi_1 h\| \leq \|\Psi_{\rm reg} h\|, \quad h \in \sE.
\]
\end{corollary}

\begin{corollary}\label{Andoun-}
Let $\Phi\in\bB(\sE,\sH)$
and $\Psi\in\bB(\sE,\sK)$.
Then the following statements are equivalent:
\begin{enumerate}\def\labelenumi{\rm(\roman{enumi})}
\item $\Psi$ admits a unique pseudo-orthogonal Lebesgue type decomposition with respect to $\Phi$;
 \item $\sD(\Phi, \Psi)$ is closed.
  \end{enumerate}
 \end{corollary}

\section{Addendum}

In Definition \ref{nonorthsum} the sum $T=T_1+T_2$ in \eqref{ltdsum}
is automatically strict, i.e., $\mul T_1 \cap \mul T_2=\{0\}$,
due to Lemma \ref{nonorthsum1}.
However, the conditions (a) and  (b) in
Definition \ref{nonorthsum} may be weakened  to:
\begin{itemize}
\item[(a)] $\ran T_1 \subset \sX$ and $\ran T_2   \subset \sY$;
\item[(b')] for every $\{f,g\} \in T$ there exist  elements
$\{f,g_1\} \in T_1$, $\{f,g_2\} \in T_2$ with $g=g_1+g_2$,
such that
\[
  \|g\|_\sK^2=\|g_1\|_\sX^2 + \|g_2\|_\sY^2,
\]
and, in addition, $\mul T_1 \cap \mul T_2 \subset XY\, \mul T$.
\end{itemize}
Then  the conditions (a) and (b')  are equivalent to (a) and (b)
in Definition \ref{nonorthsum} if and only if
the sum $T=T_1+T_2$ is strict.
Observe that the additional condition $\mul T_1 \cap \mul T_2 \subset XY\, \mul T$ in (b')
is equivalent to $\mul T_1 \cap \mul T_2 = XY\, \mul T$; it
serves as a smoothness condition, since in general one has  the range inclusion
$\ran T_1 \cap \ran T_2 \subset \ran X^\half \cap \ran Y^\half$.
The assertions in Theorem \ref{prelem0sum} remain valid
in the weaker sense if the directness condition in \eqref{sum3} is dropped.

Similar remarks can be made for the linear relations generated
by a pair of bounded operators.
Recall that the condition $\Psi(\ker \Phi)=\Psi_1(\ker \Phi) + \Psi_2(\ker \Phi)$
in Lemma \ref{uiteind}
is equivalent to the identity \eqref{telaat}.
 The conditions (a) and (b) in Remark \ref{uiteind2} may be weakened to:
\begin{itemize}\def\labelenumi{\rm(\alph{enumi})}
\item[(a)] $\ran \Psi_1 \subset \sX$ and $\ran \Psi_2 \subset \sY$;
\item[(b')]  for all $\eta \in \sE$
there exist elements $\eta', \eta'' \in \sE$
with $\Phi \eta=\Phi \eta'=\Phi \eta''$ and $\Psi \eta =\Psi_1 \eta'+\Psi_2 \eta''$,
such that
\[
 \|\Psi \eta\|_\sK^2=\|\Psi_1 \eta' \|_\sX^2 + \|\Psi_2 \eta'' \|_\sY^2,
\]
while the inclusion $\Psi_1(\ker \Phi) \cap \Psi_2(\ker \Phi) \subset XY\, \Psi(\ker \Phi)$ holds.
\end{itemize}
Then  the conditions (a) and (b')  are equivalent to (a) and (b)
in Remark \ref{uiteind2}
if and only if $\Psi_1(\ker \Phi) \cap \Psi_2(\ker \Phi)=\{0\}$.
In the weaker case the assertions in Theorem \ref{mainab0}
have to be adapted accordingly.
The directness condition \eqref{comsum1aw0} must be dropped
and the condition \eqref{comsum1cw0}
must be replaced by: for every $\eta \in \sE$ there exist $\eta', \eta'' \in \ker \Phi$, such that
\[
 \Psi_1 \eta= (I-K)\Psi (\eta-\eta') \quad \mbox{and} \quad \Psi_2 \eta=K \Psi (\eta-\eta'').
\]

\end{document}